\documentclass{amsart}
\usepackage{amsmath,amssymb, amsthm, tikz}
\usepackage{pgfkeys}
\usepackage{graphics}
\usepackage{enumitem}
\usepackage{ulem}

\usepackage[colorlinks=true, pdfstartview=FitV, linkcolor=blue, citecolor=blue, urlcolor=blue]{hyperref}

\newtheorem{lemm}{Lemma}[section]
\newtheorem{theorem}{Theorem}[section]

\newtheorem{lemma}{Lemma}[section]
\newtheorem{cor}{Corollary}[section]

\newtheorem{propos}[lemm]{Proposition}

\newenvironment{defi}{\medskip\noindent{\sc
Definition}. }{\goodbreak\medskip}

\newenvironment{nota}{\medskip\noindent{\sc
Notation}.}{\goodbreak\medskip}
\newenvironment{remk}{\noindent{\sc
Remark}. }{\goodbreak\vskip10pt}

\newenvironment{notas}{\medskip\noindent{\sc
Notations}. }{\goodbreak\medskip}

\newenvironment{exa}{\noindent{\sc
Example}. }{\goodbreak\vskip10pt}
\newenvironment{ques}{\noindent{\sc
Question}. }{\goodbreak\vskip10pt}

\def\P{{\mathcal P}}

\def\cb{{\mathcal B}}

\def\cs{{\mathcal S}}

\def\cC{{\mathcal C}}

\def\ci{{\mathcal I}}
\def\cii{{\overline{\mathcal I}}}

\def\cg{{\mathcal G}}
\def\pg{\mathcal{PG}}

\def\cm{{\mathcal M}}

\def\cu{{\mathcal U}}

\def\wt{\widetilde\Theta}

\def\R{\mathbb{R}}
\def\A{\mathbb{A}}
\def\Z{\mathbb{Z}}
\def\N{\mathbb{N}}

\def\T{\mathbb{T}}
\def\Q{\mathbb{Q}}

\def\x{\xi}

\def\smallskip{\par\vspace{1mm}}
\def\medskip{\par\vspace{2mm}}
\def\bigskip{\par\vspace{3mm}}

\newcount\equationcount
\def\thenumber{0}
\def\eq#1{\global\advance\equationcount by 1
   \def\thenumber{\number\equationcount}
                        {$$#1\eqno(\thenumber)$$}}

\tikzset{
xmin/.store in=\xmin, xmin/.default=-1.5, xmin=-1.5,
xmax/.store in=\xmax, xmax/.default=7.5, xmax=7.55,
ymin/.store in=\ymin, ymin/.default=-0.75, ymin=-0.75,
ymax/.store in=\ymax, ymax/.default=3.25, ymax=3.25,
}

\begin{document}

\title[T-weak K.A.M.]{ {Weak K.A.M. solutions and minimizing orbits of twist maps.}}

\author{Marie-Claude Arnaud$^{\dag,\ddag}$, Maxime Zavidovique$^{*,**}$}

\email{marie-claude.arnaud@math.univ-paris-diderot.fr \\ maxime.zavidovique@upmc.fr}

\date{}



\thanks{$\dag$ Universit\' e de Paris Cit\' e, Sorbonne Universit\' e, CNRS, Institut de Math\' ematiques de Jussieu-Paris Rive Gauche,
F-75013 Paris, France } 
\thanks{$\ddag$ member of the {\sl Institut universitaire de France.}}
\thanks{ $*$ Sorbonne Universit\' e, Universit\' e de Paris Cit\' e, CNRS, Institut de Math\' ematiques de Jussieu-Paris Rive Gauche,
F-75005 Paris, France}
\thanks{ $**$ financ\' e par une bourse PEPS du CNRS}

\begin{abstract} 
 For exact symplectic twist maps of the annulus, we etablish a choice of weak K.A.M. solutions $u_c=u(\cdot, c)$  that depend in a Lipschitz-continuous way on the cohomology class $c$. This allows us to make a bridge between weak K.A.M. theory of Fathi, Aubry-Mather theory for semi-orbits as developped by Bangert and existence of backward invariant pseudo-foliations as seen by Katnelson \& Ornstein. We deduce a very precise description of the pseudographs of the weak K.A.M. solutions and many interesting results as
\begin{itemize}
\item  the Aubry-Mather sets are contained in pseudographs that are vertically ordered by their rotation numbers;
\item on every image of a vertical of the annulus, there is at most two points whose negative orbit is minimizing with a given rotation number;
\item all the corresponding pseudographs are filled by minimizing semi-orbits and we provide a description of a smaller selection of full pseudographs whose union contains all the minimizing orbits; 
\item there exists an exact symplectic twist map that has a minimizing negative semi-orbit that is not contained in the pseudograph of a weak K.A.M. solution.
\end{itemize}

\end{abstract}


\maketitle
\section{Introduction and Main Results.}\label{SecIntro}
In the 80s, Aubry  and Mather elaborated a deep theory describing the dynamics of an exact symplectic twist diffeomorphism (ESTwD) of the 2-dimensional annulus restricted to the union of its minimizing orbits \cite{Aubry,Mat1}. Twenty five years later, Katznelson and Ornstein introduced a notion of pseudograph that allowed them to reprove in a geometric way some part of Aubry-Mather theory as well as a theorem of Birkhoff, \cite{KatOrn} . 

Meanwhile, Fathi made a striking connection between Aubry-Mather theory for Hamiltonian dynamical systems and the PDE approach of Hamilton-Jacobi equation. His weak K.A.M. solutions also define pseudographs. But it seems that an in depth study of weak K.A.M. solutions in the context of exact symplectic twist diffeomorphism has little been done.

Here, we fill that gap and give a precise description of weak K.A.M. solutions for  an ESTwD. Also, we revisit a theory for minimizing semi-orbits, developed by Bangert \cite{Ban2}, in the spirit of Aubry-Mather results on minimizing full orbits. Our approach  is based on a method of Lipschitz selection of weak K.A.M. solutions that we elaborate. Among other results, we prove that 
\begin{itemize}
\item our selection of full pseudographs of weak K.A.M. solutions is a vertically ordered filling of the whole annulus; all the corresponding pseudographs are filled by minimizing semi-orbits and we provide a description of a smaller selection of full pseudographs whose union contains all the minimizing orbits; 
\item every minimizing semi-orbit has a rotation number\footnote{This is already proved in \cite{Ban2}.}. These semi-orbits are  vertically arranged by their rotation numbers;

\item for a fixed rotation number, every twisted vertical\footnote{ This refer tho the forward image of a vertical.} contains at most two minimizing semi-orbit having this rotation number and every vertical contains at least one minimizing semi-orbit having this rotation number; 
\item ultimately, we provide a detailed description of the pseudographs of the weak K.A.M. solutions, especially in case of a rational rotation number,  see Proposition \ref{ordrerat}.
\end{itemize}

\subsection{Main results}
In this article we study weak K.A.M. solutions and infinite minimizing orbits of Exact Symplectic Twist Diffeomorphisms (ESTwDs in short). Along the way, we recover classical results of Aubry,  Mather and Bangert with a more weak K.A.M. approach. The aim of this paper is to be as much self-contained as can be, only the most basic results of Aubry-Mather theory for twist maps are used.

We recall briefly that\footnote{ Precise definitions will be given later.} 
\begin{itemize}
\item  there is natural variational setting for the ESTwDs: a {\em generating function} can be associated to   a ESTwD as well as an action and minimizing orbits are minimizers of this action.
\item a 1-parameter family $(T^c)_{c\in\R}$ of variational operators is defined on the set $C^0(\T, \R)$ of continuous functions on $\T$ whose fixed points are called  {\em weak K.A.M. solutions}; 
\item then, if $u$ is a fixed point of $T^c$, the associated {\em pseudograph}, which  is the partial graph $\cg(c+u')$ of $c+u'$, is backward invariant by the ESTwD; the corresponding parameter is called the cohomology class; the corresponding 
{\em full pseudograph} $\P\cg(c+u')$ is an essential curve\footnote{ This refers to a simple loop that is not isotopic to a point.}  that is the union of  $\cg(c+u')$ and some vertical segments.

 \end{itemize}
 
The following   result is reminiscent of Aubry-Mather theory for two-sided minimizing orbits (see \cite{Ban,Ban2}):  on every twisted vertical, there are at most two points with a fixed rotation number and whose negative orbit is minimizing. If $\theta\in \T$, we set $V_\theta = \{\theta\} \times \R$.  In all the article, $\Z_-$ will refer to the set of nonpositive integers. The beginning of the following already appears in Bangert \cite{Ban2}:


\begin{theorem}\label{Ttwistvertical}  Let $f$ be a $C^1$ ESTwD of $\T\times \R$. Then every negative minimizing orbit has a rotation number.
Let $\theta\in \T$ and $\rho_0 \in \R$, then
\begin{itemize}
\item if $\rho_0\notin \Q$, there exists at most one $(x,p)\in f(V_\theta)$ such that $\big(\pi_1\circ f^i(x,p)\big)_{i\in \Z_-}$ is minimizing with rotation number $\rho_0$;
\item if $\rho_0\in \Q$, there exists at most two $(x,p)\in f(V_\theta)$ such that $\big(\pi_1\circ f^i(x,p)\big)_{i\in \Z_-}$ is minimizing with rotation number $\rho_0$.
\end{itemize}
\end{theorem}

 Now we state the existence of a Lipschitz continuous choice of fixed point $u_c$ of $T^c$, that generates a continuous  and ordered choice of the associated pseudograph.
 \begin{theorem}\label{Tgenecont}
Let $f$ be a $C^1$ ESTwD of $\T\times \R$. Then there exists a  continuous map $u:\T\times \R\rightarrow \R$ such that
\begin{enumerate}
\item $u(0,c)=0$;
\item  the map $(\theta, c)\mapsto \frac{\partial u_c}{\partial \theta}(\theta)$ is continuous on its set of definition; 
\item each $u_c= u(\cdot ,c)$ is a weak K.A.M. solution for the cohomology class $c$, this implies that:
\begin{itemize}
\item each $u_c= u(\cdot ,c)$ is semi-concave (hence  derivable almost everywhere)\footnote{The definition of a semi-concave function is given in subsection \ref{ssweakk}.};
 \item each   partial graph $\mathcal{G}(c+u_c')$ of $c+\frac{\partial u_c}{\partial \theta}$ is backward invariant by $f$;
 \item {the negative orbit $(f^{-n}(\theta, r))_{n\geq 0}$ of every point $(\theta, r)\in \cg(c+u'_c)$ is minimizing;}
  \end{itemize}
 
  \item\label{croissant} {    for all $c\leqslant c'$,     we have  $c+u'_c(\theta) \leqslant c'+u'_{c'}(\theta)$ at all $\theta\in \T$ where both derivatives exist;}
\item\label{u-lip}{the function $u$ is locally Lipschitz continuous (and even $1$--Lipschitz with respect to $c$).}
\end{enumerate}
\end{theorem}
 From Theorems \ref{Tgenecont} and \ref{Ttwistvertical}, we deduce that the negative orbits of  the  points of $\cg(c+u'_c)$ have a unique rotation number that we denote by $\rho(c)$.

{ Next Theorem explains that  the associated full pseudographs make a vertically ordered continuous filling of the whole annulus.}

\begin{theorem}\label{Tpseudofol}
 With the notations of Theorem \ref{Tgenecont}, we have
 \begin{enumerate}
 \item the map $c\mapsto  \mathcal{PG} (c+u'_c)$ is continuous for the Hausdorff topology; 
 \item $\displaystyle { \bigcup_{c\in \R} \mathcal{PG} (c+u'_c)= \A;}$
 \item if 
{  $\rho(c)<\rho(c')$}, then for all $(q,p)\in \mathcal{PG} (c+u'_c)$ and $(q, p')\in \mathcal{PG} (c+u'_{c'})$, we have $p<p'$.
  \end{enumerate} 
 \end{theorem}
As a result of the proof, we will deduce (see Proposition \ref{Pordreirrat}) that the Aubry-Mather\footnote{The definition of Aubry-Mather set is given in subsection \ref{ssAM}.} sets are contained in pseudographs that are vertically ordered by their rotation numbers. 

{ The next statement explains  that the weak K.A.M. solutions reflect all the richness of negative minimizing semi-orbits.}
 \begin{theorem} \label{solK.A.M.calib} {With the notations of Theorem \ref{Tgenecont}, }
let  $(\theta_i,r_i)_{i\in \Z_-}\in \A^{\Z_-}$ be a minimizing negative orbit of $f$, then there exist $c\in \R$ and a weak K.A.M. solution $u_c : \T \to \R$ at cohomology $c$  such that 
$$(\theta_i,r_i)_{i\in \Z_-} \subset \mathcal{PG} (c+u'_{c}),$$
$$(\theta_i,r_i)_{i<0} \subset \mathcal{G} (c+u'_{c}).$$

\end{theorem}

Jean-Pierre Marco raised the following question. 

\begin{ques} If $(\theta_i,r_i)_{i\in \Z_-}\in \A^{\Z_-}$ is a minimizing negative semi-orbit of $f$, is it necessarily  contained in $\overline{\mathcal{G} (c+u'_{c})}$?\end{ques} In part \ref{SSJPMarco}, we answer negatively to this question and provide an example where a minimizing negative semi-orbit is not contained in such a set.

{ Finally, we prove that we can use only a particular subset of $\{ \cg (c+u'_c); c\in\R\}$ to recover the union of all the pseudographs of weak K.A.M. solutions.}
\begin{theorem}\label{TPlacesolK.A.M.} { For every $\rho_0\in\R$, $\rho^{-1}(\{\rho_0\})$ is a segment $[a, b]$. With the notations of Theorem \ref{Tgenecont},}
if  $c\in [a,b]$ and $u$ is a weak K.A.M. solution at cohomology $c$, then 
$$\mathcal G (c+ u') \subset \mathcal G (a+ u'_a)\cup \mathcal G (b+ u'_b),$$
and by taking closures:
$$\overline{\mathcal G (c+ u')} \subset \overline{ \mathcal G (a+ u'_a)} \cup \overline{\mathcal G (b+ u'_b)}.$$
{ More precisely, 
\begin{itemize}
\item when $\rho_0$ is irrational, $\rho^{-1}(\{\rho_0\})$ is a single point;
\item when $\rho_0$ is rational, the union of  two pseudographs  contain all pseudographs with this rotation number. Moreover, those two pseudographs intersect along minimizing periodic orbits.
\end{itemize}
 Moreover, every minimizing semi-orbit $(\theta_i,r_i)_{i\in \Z_-}$ with rotation number $\rho_0$ is contained in 
$$\mathcal{PG} (a+ u'_a)\cup \mathcal{PG} (b+ u'_b).$$

}
\end{theorem}
In fact, we will provide a more precise description of how the full pseudographs $\pg(a+u'_a)$ and $\pg(b+u'_b)$ are positioned and of the way $\pg(c+u')$ is built by taking pieces of  $\pg(a+u'_a)$ and $\pg(b+u'_b)$ and gluing them with vertical segments.

Once we have proved that  there always exists a continuous choice $u(\theta, c)$ of  weak K.A.M. solutions, we wonder when $u$ can be more regular. We  recall that an ESTwD is said to be $C^0$-integrable if the annulus $\T\times \R$ is $C^0$-foliated by $C^0$ invariant graphs. 
\begin{theorem}\label{Tgeneder}
With the notations of Theorem \ref{Tgenecont}, we have equivalence of 
\begin{enumerate}
\item\label{pt1Tgeneder} $f$ is $C^0$-integrable; 
\item\label{pt2Tgeneder}  the map $u$ is  $C^1$. 
\end{enumerate} Moreover, in this case, $u$ is unique and  we have\footnote{See the notation $\pi_1$ at the beginning of subsection \ref{ssnota}.}
 \begin{itemize}
 \item the graph of $c+u_c'$ is a leaf of the invariant foliation;
 \item $h_c:\theta \mapsto \theta+\frac{\partial u}{\partial c}(\theta,c)$ is a semi-conjugation between the projected Dynamics $g_c: \theta\mapsto \pi_1\circ f\big(\theta, c+\frac{\partial u}{\partial \theta}(\theta,c)\big)$ and a rotation $R$ of $\T$, i.e. $h_{c}\circ g_c=R\circ h_{c}.$
   \end{itemize}
 \end{theorem}
We will prove here the implication \eqref{pt2Tgeneder} $\Rightarrow$ \eqref{pt1Tgeneder}. The reverse implication is addressed in the companion paper \cite{ArZav}.

\subsection{A double pendulum} Let us illustrate some of our results on a simple example. Let $H : \T\times \R$ be the Hamiltonian  defined by $(\theta, p) \mapsto  \frac12 |p|^2+ \cos(4\pi \theta)$ and $f= \phi_H^{t_0} : \A \to \A$ be the Hamiltonian flow of $H$ for a small time $t_0$. Then it is known that for small enough $t_0$, $f$ is an ESTwD. Moreover, weak K.A.M. solutions for $H$ and $f$ can be proven to be the same.
 
With that in mind, we obtain that for $\rho_0 = 0$, then $\rho^{-1}(\{0\}) = [-a,a]$, where $a=\int_0^1 \sqrt{2-2\cos(4\pi\theta)} d\theta$. The integrated function $\theta \mapsto  \sqrt{2-2\cos(4\pi\theta)}$,  denoted $f^+$, corresponds to  the upper part of the level set $H^{-1}(\{1\})$. The lower part is the graph of $-f_+$.

The unique (up to constants) weak K.A.M. solution $u_a$, at cohomology $a$  is $C^1$ and such that $a+u'_a$ is the graph on $f_+$, in  blue in figure \ref{double1}. Similarly, the unique (up to constants) weak K.A.M. solution $u_{-a}$, at cohomology $-a$  is $C^1$ and such that $-a+u'_{-a}$ is the graph on $-f_+$, in  red in figure \ref{double1}. Note that those two graphs intersect at the only minimizing fixed points of $f$ that are of coordinates $(0,0)$ and $(\frac12,0)$. This fact will be generalized in Proposition \ref{ordrerat}.

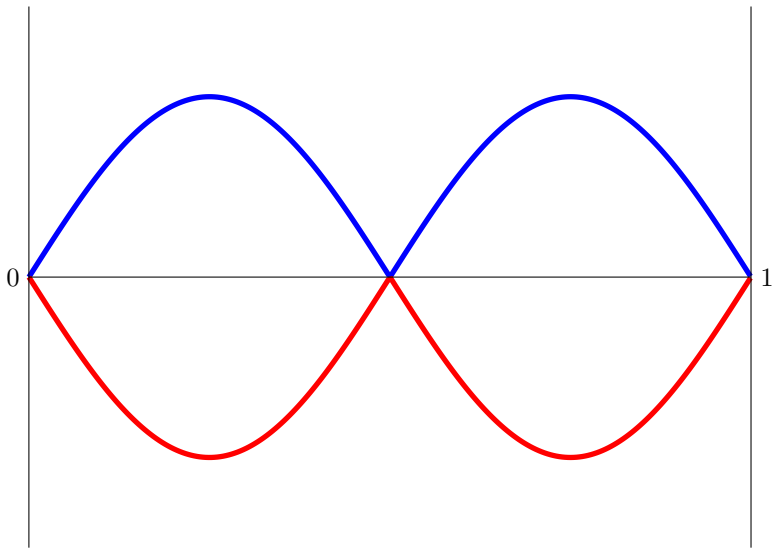
\begin{figure}[h!]
\begin{center}

 \begin{tikzpicture}[scale = 1.2]
    \draw[color=blue][line width=2pt][domain=0:4][samples=250] plot (2
   *\x,2* abs{sin(pi*\x/2 r)} );
    \draw[color=red][line width=2pt][domain=0:4][samples=250] plot (2
   *\x,-2* abs{sin(pi*\x/2 r)} );
   
   \draw  (0,-3)--(0,3) ;
     \draw  (8,-3)--(8,3) ;
     \draw (0,0) -- (8,0);

\draw (0,0) node[left]{0} ;
\draw (8,0) node[right]{1} ;
\end{tikzpicture}

\caption{The level set $H^{-1}(\{1\})$ is the union of the graphs of $a+u'_a$ in blue and $-a + u'_{-a}$ in red.}
\label{double1}
\end{center}
\end{figure}

Let us now focus at weak K.A.M. solutions at cohomology $0$. Their derivative lie in  $H^{-1}(\{1\})$. As weak K.A.M. solutions are semi--concave the derivative can only jump downwards and must have vanishing integral. So it looks like the red part in figure \ref{double3}.

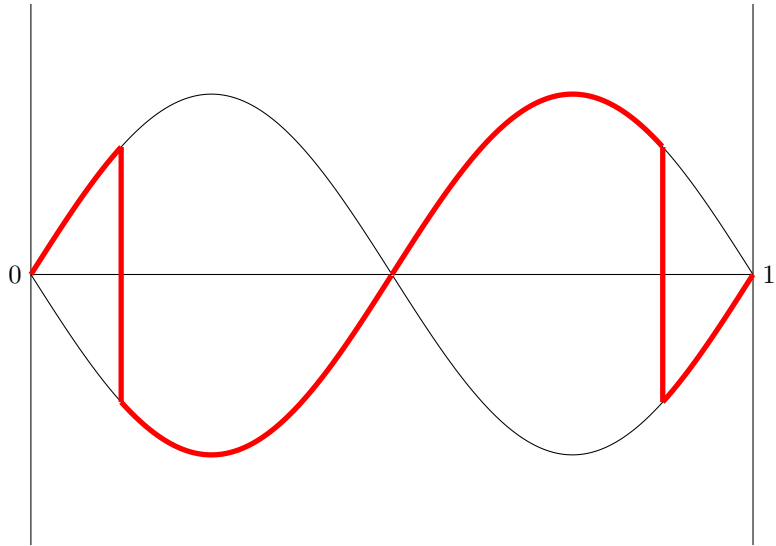
\begin{figure}[h!]
\begin{center}

  \begin{tikzpicture}[scale = 1.2]
    \draw[domain=0:4][samples=250] plot (2
   *\x,2* abs{sin(pi*\x/2 r)} );
    \draw[domain=0:4][samples=250] plot (2
   *\x,-2* abs{sin(pi*\x/2 r)} );
   
   \draw  (0,-3)--(0,3) ;
     \draw  (8,-3)--(8,3) ;
     \draw (0,0) -- (8,0);

\draw (0,0) node[left]{0} ;
\draw (8,0) node[right]{1} ;

  \draw[color=red][line width=2pt][domain=0:0.5][samples=250] plot (2
   *\x,2* abs{sin(pi*\x/2 r)} );
    \draw[color=red][line width=2pt][domain=1/2:2][samples=250] plot (2
   *\x,-2* abs{sin(pi*\x/2 r)} );

  \draw[color=red][line width=2pt][domain=2:3.5][samples=250] plot (2
   *\x,2* abs{sin(pi*\x/2 r)} );
    \draw[color=red][line width=2pt][domain=3.5:4][samples=250] plot (2
   *\x,-2* abs{sin(pi*\x/2 r)} );

\draw [color=red] [line width=2pt](1,2* abs{sin(pi*1/4 r)})--(1,-2* abs{sin(pi*1/4 r)}) ;
\draw [color=red] [line width=2pt](7,2* abs{sin(pi*1/4 r)})--(7,-2* abs{sin(pi*1/4 r)}) ;

\end{tikzpicture}

\caption{The full pseudograph of a weak K.A.M. solution at cohomology $0$ in red.}
\label{double3}
\end{center}
\end{figure}
\newpage 

The construction that we propose to prove Theorem \ref{Tgenecont} respects the $\frac12$ periodicity of $H$, hence the weak K.A.M. solution obtained is itself $\frac12$-periodic as shown in the next figure \ref{double2}.

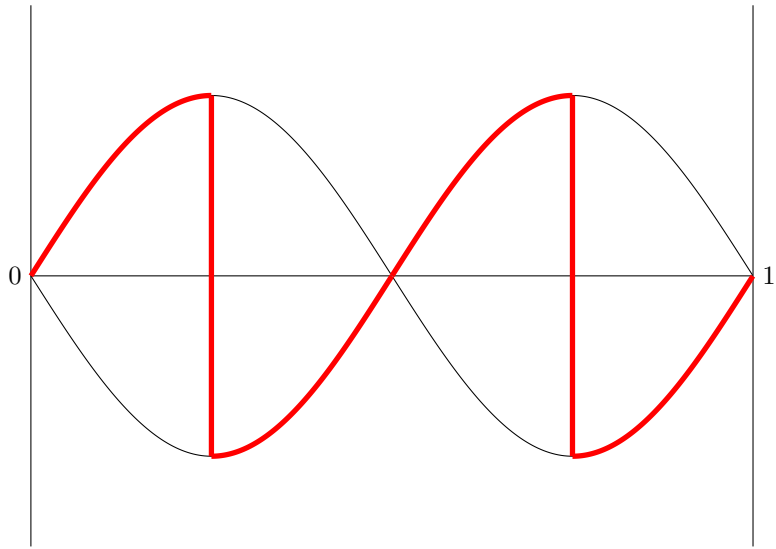
\begin{figure}[h!]
\begin{center}

  \begin{tikzpicture}[scale = 1.2]
    \draw[domain=0:4][samples=250] plot (2
   *\x,2* abs{sin(pi*\x/2 r)} );
    \draw[domain=0:4][samples=250] plot (2
   *\x,-2* abs{sin(pi*\x/2 r)} );
   
   \draw  (0,-3)--(0,3) ;
     \draw  (8,-3)--(8,3) ;
     \draw (0,0) -- (8,0);

\draw (0,0) node[left]{0} ;
\draw (8,0) node[right]{1} ;

  \draw[color=red][line width=2pt][domain=0:1][samples=250] plot (2
   *\x,2* abs{sin(pi*\x/2 r)} );
    \draw[color=red][line width=2pt][domain=1:2][samples=250] plot (2
   *\x,-2* abs{sin(pi*\x/2 r)} );

  \draw[color=red][line width=2pt][domain=2:3][samples=250] plot (2
   *\x,2* abs{sin(pi*\x/2 r)} );
    \draw[color=red][line width=2pt][domain=3:4][samples=250] plot (2
   *\x,-2* abs{sin(pi*\x/2 r)} );

\draw [color=red] [line width=2pt](2,2* abs{sin(pi*1/2 r)})--(2,-2* abs{sin(pi*1/2 r)}) ;
\draw [color=red] [line width=2pt](6,2* abs{sin(pi*1/2 r)})--(6,-2* abs{sin(pi*1/2 r)}) ;

\end{tikzpicture}

\caption{The full pseudograph of the weak K.A.M. solution at cohomology $0$ selected by the construction in Theorem \ref{Tgenecont} in red.}
\label{double2}
\end{center}
\end{figure}

\subsection{Further comments and related results}

\begin{itemize}
\item The beginning of Theorem \ref{Ttwistvertical}  is already present in Bangert's \cite{Ban2}. Actually, Bangert proves it for a more general setting that contains finite compositions of ESTwDs. However, restricting to an ESTwD allows us to obtain the two final items of Theorem \ref{Ttwistvertical} that would not hold in Bangert's setting. The same happens for other geometrical results as Proposition \ref{imageWK.A.M.} that holds for ESTwDs but not in Bangert's more general setting.

 From a methodology point of view, Bangert uses Buseman functions, that are functions defined on $\R$. We rather focus on weak K.A.M. solutions that are defined on the circle  $\T$ and have been widely studied  in recent years. This allows for proofs that we hope more accessible to people familiar with weak K.A.M. theory. Moreover as already explained in the Introduction, this draws a parallel with a variational approach to Katznelson and Orstein's results on backward invariant pseudographs. 
\item{\bf Theorem \ref{Tgenecont}}  selects in a continuous way a unique weak K.A.M. solution $u_c$ for every cohomology class $c\in\R$. Let us mention two related results.
\begin{itemize}
\item The recent works in \cite{DFIZ1} for the autonomous case and in  \cite{DFIZ2} and \cite{SuThi} for the discrete case select a unique solution, called discounted solution, for every cohomology class. We give in Appendix \ref{sscomplint} an example of a $C^\infty$ integrable ESTwD (coming from an autonomous Tonelli Hamiltonian)   for which the discounted method doesn't select a transversally continuous weak K.A.M. solution.  Hence our method is different from the discounted one.
\item If we have not a unique choice of a weak K.A.M. solution for  every cohomology class $c\in H^1(M, \R)$, we cannot speak of $C^1$ regularity with respect to $c$ for the map $c\mapsto \{u_c\}$ that sends $c$ to the whole set of weak K.A.M. solutions of cohomology class $c$. Observe nevertheless that a kind of local Lipschitz regularity was studied in \cite{LYY} (for weak K.A.M. solutions for Tonelli Hamiltonians) with no uniqueness.

\item Around the same time this research was done, similar results were established in \cite{Zhang}. However our results are more precise with some respect (for instance the Lipschitz selection of weak K.A.M. solutions). Moreover, our study and description of weak K.A.M. solutions has not been obtained elsewhere.
\end{itemize}

\item {\bf Theorem \ref{Tpseudofol}}   compares the  cohomology classes of pseudographs that correspond to distinct rotation numbers. In the setting of Hamiltonian flows with two degrees of freedom, an analogous statement is proved in \cite{CheXu} concerning the cyclic order of rotation and cohomology vectors.

\item Similarly to Theorem  \ref{Tpseudofol},   Katznelson \& Ornstein provide in \cite{KatOrn} a continuous covering of the annulus by full pseudographs. 

\end{itemize}

\subsection{Content of the different sections}
 We chose to present our results  in an order other than the order of the proofs.
 
To prove all these results, we will use together Aubry-Mather theory, weak K.A.M. theory in the discrete case. Let us detail what will be in the different sections
\begin{itemize}
\item Section \ref{secth11} contains some reminders on ESTwDs,  Aubry-Mather theory, on discrete weak K.A.M. theory, some new results on the weak K.A.M. solutions and the proof of   Theorems \ref{Tgenecont}  and \ref{Tpseudofol};
\item the second implication of Theorem \ref{Tgeneder} is proved in section \ref{secsecond};
\item results on minimizing sequences and weak K.A.M. solutions are stated and proved in section \ref{sec4},  where we prove Theorems \ref{Ttwistvertical}, \ref{solK.A.M.calib} and  \ref{TPlacesolK.A.M.};
 \item Appendices \ref{AppA}  contains some examples, Appendix \ref{Apfulpseudo} deals with full pseudographs,  Appendix \ref{appendix-3} explains a point that is useful to prove Theorem \ref{Tgenecont}.
\end{itemize}

\subsection*{Acknowledgements} The authors are grateful to  Fr\' ed\' eric Le Roux for insightful discussions that helped clarify and simplify some proofs of this work and to Jean-Pierre Marco for asking them intriguing  questions.



 \section{ Aubry-Mather and weak K.A.M. theories for ESTwDs and proof of Theorems \ref{Tgenecont} and \ref{Tpseudofol}}\label{secth11}
 \subsection{The setting}\label{ssnota}  The definitions and results that we give here are very classical now. Good references are \cite{ForMat, Gol1,Mat2, Mos, Be2,Man2}.
 
 Let us introduce some notations
 \begin{notas}
 \begin{itemize}
 \item $\T=\R/\Z$ is the circle and $\A=\T\times \R$ is the annulus ;  $\pi: 
 \R\rightarrow \T$ is the usual projection;
 \item the universal covering of the annulus is denoted by $p:\R^2\rightarrow \A$;
 \item the corresponding projections are $\pi_1: (\theta, r)\in \A\mapsto \theta\in \T$ and $\pi_2: (\theta, r)\in \A\mapsto r\in \R$; we denote also the corresponding projections of the universal covering by $\pi_1$, $\pi_2~: \R^2\rightarrow \R$;
 \item the Liouville 1-form is defined on $\A$ as being $\lambda=\pi_2d\pi_1=rd\theta$; then $\A$ is endowed with the symplectic form $\omega=-d\lambda$.
 \end{itemize}
\end{notas}

  Let us give  the definition of an exact symplectic twist diffeomorphism.
\begin{defi}
An {\em exact symplectic  twist diffeomorphism} (in short ESTwD)  $f:\A\rightarrow \A$ is a $C^1$ diffeomorphism such that
\begin{itemize}
\item $f$ is isotopic to identity;
\item $f$ is exact symplectic, i.e. if $f(\theta, r)=(\Theta, R)$, then the 1-form $Rd\Theta-rd\theta$ is exact;
\item $f$ has the {\em twist property} i.e. if $F=(F_1, F_2):\R^2\rightarrow \R^2$ is any lift of $f$, for any $\tilde \theta\in \R$, the map $r\in \R\mapsto F_1(\tilde \theta, r)\in \R$ is an increasing $C^1$ diffeomorphism from $\R$ onto $\R$.
\end{itemize}
\end{defi}

A $C^2$ generating function $S:\R\times\R\rightarrow \R$   that satisfies the following definition can be associated  to any lift $F$ of such an ESTwD $f$.
 
\begin{defi}
The $C^2$ function $S:\R^2\rightarrow \R$ is a {\em generating function} of the lift $F:\R^2\rightarrow \R^2$ of an ESTwD  if
\begin{itemize}
\item $S(\tilde \theta+1,\widetilde \Theta+1)=S(\tilde\theta,\widetilde \Theta)$;
\item $\displaystyle{\lim_{|\widetilde\Theta-\tilde\theta|\rightarrow \infty}\frac{S(\tilde\theta,\widetilde \Theta)}{|\widetilde\Theta-\tilde\theta|}=+\infty}$; we say that $S$ is {\em superlinear}; 
\item for every $\tilde\theta_0,\widetilde \Theta_0\in \R$, the maps $\tilde\theta\mapsto \frac{\partial S}{\partial\widetilde \Theta}(\tilde\theta,\widetilde \Theta_0)$ and $\widetilde\Theta\mapsto \frac{\partial S}{\partial\tilde \theta}(\tilde\theta_0, \widetilde\Theta)$ are decreasing diffeomorphisms of $\R$;
\item for $(\tilde\theta, r), (\widetilde\Theta, R)\in \R^2$, we have the following  equivalence 
\begin{equation}\label{genfundef}
F(\tilde\theta, r)=(\widetilde\Theta, R)\Leftrightarrow r=-\frac{\partial S}{\partial \tilde\theta}(\tilde\theta,\widetilde \Theta)\quad{\rm and}\quad R=\frac{\partial S}{\partial\widetilde \Theta}(\tilde\theta,\widetilde \Theta).
\end{equation}
\end{itemize}
\end{defi}

 \begin{remk} 
 J.~Moser proved in \cite{Mos} that such an ESTwD  is the time 1 map of a $C^2$ 1-periodic in time Hamiltonian $H:\T\times\R\times\R\rightarrow \R$  that is $C^2$ convex in the  fiber direction\footnote{In fact J.~Moser assumed that $f$ is smooth.}, i.e. such that $$\frac{\partial^2 H}{\partial r^2}(\theta, r, t)>0.$$
 Then there exists a relation between the Hamiltonian that was built by J.~Moser and the generating function. Indeed, if we denote by $(\Phi_t)$ the time $t$ map of the Hamiltonian $H$ that is defined on $\R^2$ and by $L$ the associated Lagrangian that is defined by
 $$L(\tilde\theta, v, t)=\max_{r\in\R}\big(rv-H(\tilde\theta, r, t)\big),$$
then we have
 \begin{itemize}
 \item for every $t\in (0, 1]$, $\Phi_t$ is an ESTwD and $\Phi_1=F$;
 \item there exists a $C^1$ time-dependent family of $C^2$ generating functions $S_t$ of $\Phi_t$ such $S_1=S$ and for all $ (\tilde\theta, r), (\widetilde\Theta, R)\in\R^2,$
 $$ \Phi_t(\tilde\theta,r)=(\widetilde\Theta, R)\Rightarrow S_t(\tilde\theta,\widetilde \Theta)=\int_0^tL\big(\pi_1\circ \Phi_s(\tilde\theta, r), \frac{\partial}{\partial s}\big( \pi_1\circ \Phi_s(\tilde\theta, r)\big),s\big)ds.$$
 \end{itemize}
 In other words, the generating function is also the Lagrangian action.
 \end{remk}
 \subsection{Aubry-Mather theory}\label{ssAM} Good references for what is in this section are \cite{Ban}, \cite{Gol1} and \cite{Arna3}.
 Let us recall the definition of some particular invariant sets.
 \begin{defi} Let $F:\R^2\rightarrow \R^2$ be a lift of an ESTwD  $f$.
 \begin{itemize}
 \item a subset $E$ of $\R^2$ is {\em well-ordered} if it is invariant under the translation $(\tilde\theta, r)\mapsto (\tilde\theta+1, r)$ and $F$ and if for every $x_1, x_2\in E$, we have
 $$\big[ \pi_1(x_1)<\pi_ 1(x_2)\big]\Rightarrow \big[ \pi_1\circ F(x_1)<\pi_1\circ F(x_2)\big];$$
 this notion is independent from the lift of $f$ we use;
 \item a subset $E$ of $\A$ is {\em well-ordered} if $p^{-1}(E)$ is well-ordered;
 \item an {\em Aubry-Mather set } for $f$ is a compact well-ordered set or the lift of such a set;
  \item a piece of orbit $(\tilde\theta_k, r_k)_{k\in [a, b]}$ for $F$ is {\em minimizing} if for every sequence  $ (\tilde\theta'_k)_{k\in [a, b]}$ with $\tilde\theta_a=\tilde\theta'_a$ and $\tilde\theta_b=\tilde\theta'_b$, it holds
 $$\sum_{j=a}^{b-1}S(\tilde\theta_j,\tilde \theta_{j+1})\leqslant \sum_{j=a}^{b-1}S(\tilde\theta'_j,\tilde \theta'_{j+1});$$
 then we say that $(\tilde\theta_j)_{j\in [a, b]}$ is a {\em minimizing sequence} or {\em segment};
 \item {  an infinite piece of orbit, or a full orbit for $F$ is minimizing if all its finite subsegments are minimizing;}
 \item an invariant set is said to be minimizing if all the orbits it contains are minimizing.
 \end{itemize}
 \end{defi}
 The following properties of the well-ordered sets are well-known
 \begin{enumerate}
 \item a minimizing orbit and its translated orbits by $(\tilde\theta, r)\mapsto (\tilde\theta+1, r)$  define a well-ordered set;
 \item the closure of a well-ordered set is a well-ordered set;
 \item\label{Ptcirclehomeom} any well-ordered set $E$ is contained in the (non-invariant) graph of a Lipschitz map $\eta:\T\rightarrow \R$;  it follows that the map $N = \big(\cdot , \eta(\cdot)\big) : \T \to {\rm Graph}(\eta)$ is Lipschitz and so are the maps  $\pi_1\circ f\circ N_{ |\pi_1(E)}$ and $\pi_1\circ f^{-1}\circ N_{ |\pi_1(E)}$ . This  implies that the projected restricted Dynamics $\pi_1\circ f\big(\cdot, \eta(\cdot)\big)_{|\pi_1(E)}$ to an Aubry-Mather set is the restriction of a biLipschitz orientation  preserving circle homeomorphism;
 \item any well-ordered set $E$ in $\R^2$  has a unique rotation number $\rho(E)$ \big(the one of the circle homeomorphism we mentioned in  Point (\ref{Ptcirclehomeom})\big), i.e.
 $$\forall x\in E, \quad \lim_{k\rightarrow\pm\infty} \frac{1}{k}\big(\pi_1\circ F^k(x)-\pi_1(x)\big)=\rho(E);$$
 \item for every $\alpha\in \R$, there exists a minimizing Aubry-Mather set $E$  such that $\rho (E)=\alpha$; 
 \item  if $\alpha$ is irrational, there is a unique  minimizing Aubry-Mather that is minimal (resp. maximal) for the inclusion; the minimal one is then a Cantor set  or a complete graph and the maximal one $\cm(\alpha)$ is the union of the minimal one and orbits that are homoclinic to the minimal one;
 \item if $\alpha$ is rational, any Aubry-Mather set that is minimal for the inclusion is a periodic orbit;
 \item\label{ptgraphminimizing}  any essential invariant curve by an ESTwD  is in fact a Lipschitz graph (Birkhoff theorem, see \cite {Bir1}, \cite{Fa1}  and \cite {He1}) and a well-ordered  minimizing set.
   \end{enumerate}
 We will need more precise properties for minimizing orbits.  
 \begin{defi}{\rm Let $a=(a_k)_{k\in I}$ and $b=(b_k)_{k\in I}$ be two finite or infinite sequences of real numbers. Then
 \begin{itemize}
 \item if $k\in I$, we say that $a$ and $b$ cross at $k$ if $a_k=b_k$;
 \item if $k, k+1\in I$, we say that $a$ and $b$ cross between $k$ and $k+1$ if 
 $$(a_k-b_k)(a_{k+1}-b_{k+1})<0.$$
 \end{itemize}

}
\end{defi}
{Note that concerning the first item, the traditional terminology also imposes that $(a_{k-1}-b_{k-1})(a_{k+1}-b_{k+1})<0$ when $k$ is in the interior of $I$. However, due to the twist condition, this is automatic  for projections of orbits of $F$  as soon as $a_k=b_k$  if the two orbits are distinct.}

\begin{propos}\label{PAubryfund}{\bf (Aubry fundamental lemma)} 
If $(a,b,a',b')\in \R^4$ verify $(a-b)(a'-b')<0$ then 
$$S(a,a')+S(b,b')>S(a,b')+S(a',b).$$

As a consequence, two distinct minimizing sequences cross at most once except possibly at the two endpoints  when the sequence is finite.

\end{propos}
\subsection{Classical results on weak K.A.M. solutions}\label{ssweakk}
Good references are \cite{Be1},  \cite{Be2} or \cite{GT1}. 
We assume that $S$ is a generating function of a lift $F$ of an ESTwD  $f$.

We define on $C^0(\T, \R)$ the so-called {\em negative Lax-Oleinik maps} $T^c$ for $c\in \R$ as follows:

 if $ u\in C^0(\T, \R)$, we denote by $\tilde u : \R \to \R$ its lift and  
\begin{equation} \label{Edefweakk}  \forall \tilde\theta\in \R, \quad    \widetilde T^c \tilde u(\tilde\theta)=\inf_{\tilde\theta'\in\R} \big(\tilde u(\tilde\theta')+S(\tilde\theta',\tilde \theta)+c(\tilde\theta'-\tilde\theta)\big).\end{equation}
The function $ \widetilde T^c \tilde u$ is then  $1$-periodic and the negative  Lax-Oleinik operator is defined as the induced map $T^c u : \T \to \R$.

An alternative but equivalent definition is as follows (see also \cite{Za} for similar constructions): define the function 
\begin{equation}\label{projS}
\forall (\theta,\theta') \in \T\times \T, \quad S^c (\theta,\theta') = \inf_{\substack{\pi(\tilde\theta) = \theta \\ \pi(\tilde\theta')=\theta'}}  S(\tilde\theta, \tilde\theta')+c(\tilde\theta-\tilde\theta').
\end{equation}
Then
$$\forall \theta\in \T,\quad T^c u(\theta) = \inf_{\theta'\in \T} u(\theta')+S^c(\theta',\theta).$$

Then it can be proved that there exists a unique function $\alpha: \R\rightarrow \R$ such that the map $\widehat T^c=T^c+\alpha(c)$ that is defined by
$$\widehat T^c(u)=T^c(u)+\alpha(c)$$
has at least one fixed point in $C^0(\T, \R)$, i.e.  if $u\in C^0(\T, \R)$ is such a fixed point, its lift verifies 
\begin{equation} \label{Esolweakk}  \forall \tilde\theta\in \R, \quad   \tilde u(\tilde\theta) = \inf_{\tilde\theta'\in\R} \big(\tilde u(\tilde\theta')+ S(\tilde\theta',\tilde \theta)+c(\tilde\theta'-\tilde\theta)+\alpha(c)\big).\end{equation}
 
Such a fixed point is called a {\em weak K.A.M. solution}. It is not necessarily unique.  For example,  if $u$ is a weak K.A.M. solution, so is $ u+k$ for every $k\in \R$, but there can also be other solutions. We denote by $\cs_c$ the set of these weak K.A.M. solutions.  There is no link in general for solutions corresponding to distinct $c$'s.  We recall
\begin{defi}
Let $u:\R\rightarrow\R$ be a function and let $K>0$ be a constant. Then $u$ is $K$-semi-concave if for every $x$ in $\R$, there exists some $p\in \R$ so that:
$$\forall y\in \R,\quad u(y)-u(x)-p(y-x)\leqslant \frac{K}{2}(y-x)^2.$$

 A function $v :\T\to \R$ is $K$-semi-concave if  its lift $\tilde v: \R \to \R$ is.
\end{defi}
A good reference for semi-concave functions is the appendix A of  \cite{Be2} or \cite{cansin}.
\begin{nota}
If $u\in C^0(\T, \R)$ and $c\in \R$, we will denote by $\cg(c+u')$ the partial graph of $c+u'$. This is a graph above the set of derivability of $u$.\\
When $u$ is semi-concave, we sometimes say that $\cg(c+u')$ is a {\em pseudograph}.
\end{nota}

Let us end  with  definitions:
 \begin{defi}\label{Clarke}
 Let $g:\T \to \R$ be a Lipschitz function (hence derivable almost everywhere). We define 
 $$\forall x\in \T, \quad \partial g(x) = {\rm co} \big\{ (x,p)\in  \T\times\R, \ \ (x,p)\in \overline{ \cg (g')}\big\}.$$
 \end{defi}
The notation co stands for the convex hull  in the fiber direction. The sets  $\partial g(x)$ are non empty, (obviously) convex and compact. They are particular instances of the Clarke subdifferential. This set is a good candidate for a generalized derivative because if $g$ is derivable at $x$ then $\big(x,g'(x)\big)\in \partial g(x)$. Moreover, if $\partial g(x)$ is a singleton, then $g$ is derivable at $x$. The converse is in general not true, but it is however true for semi-concave functions. 

\begin{defi}
If $g : \T \to \R$ is Lipschitz and $c\in \R$, we define  $\mathcal{PG}(c+g') =\{ (0,c)+\partial g (t), \quad t\in \T\}$.  If $g$ is semi-concave, we call it the full pseudograph of $c+g'$.
\end{defi}

 A proof of the  following proposition is given in Appendix \ref{Apfulpseudo}.
\begin{propos}\label{hausdorff}
Let $(f_n)_{n\in \mathbb N} $ be a sequence of equi-semi-concave functions from $\T$ to $\R$ that converges (uniformly) to a function $f$ (that is hence also semi-concave).

Then $\mathcal{PG}(f'_n) $ converges to $\mathcal{PG}(f')$ for the Hausdorff distance.
\end{propos}

The following results can be found in the papers that we quoted

 \begin{enumerate}[label=(\alph*)]
\item\label{first} the function $\alpha$ is convex   and superlinear;
\item if $u\in C^0(\T, \R)$, then $\widehat T^cu$  is semi-concave and then differentiable Lebesgue almost everywhere; 
\item\label{ptdifftt} the function  $\widehat T_c u$ is differentiable at $x$ if and only if there is only one $y$ where the minimum is attained in Equality (\ref{Esolweakk}); in this case, if $u$ is semi-concave, then it is differentiable at $y$ and we have 
$$f\big(y, c+u'(y)\big)=\big(x, c+(\widehat T^cu)'(x)\big);$$
if $u$ is a weak K.A.M. solution for $\widehat T^c$ {that is differentiable at $x$ then} $\Big(f^{k}\big(x, c+u'(x)\big)\Big)_{k\in\Z_-}$ is a minimizing piece of orbit that is contained in $\cg(c+u')$;
\item\label{same} moreover, for any compact subset $K$ of $\R$, the weak K.A.M. solutions for $T^c$ with $c\in K$ are uniformly semi-concave (i.e. for a fixed constant of semi-concavity) and then uniformly Lipschitz;
\item\label{ptinvgraph} if $u\in C^0(\T, \R)$ is semi-concave, then 
$$f^{-1}\big(\overline{\cg(c+(\widehat T^cu)')}\big)\subset \cg(c+u');$$
if $(x,p) \in \overline{\cg(c+(\widetilde T^c\tilde u)')}$ and $\big(y, c+\tilde u'(y)\big) = F^{-1}(x,p)$ then 
$$\widetilde T^c\tilde u(x) =\tilde u(y) + S(y,x) + c(y-x);$$ 
if $u$ is a weak K.A.M. solution for $\widehat T^c$, then $\cg(c+u')$ satisfies
$$f^{-1}\big(\overline{\cg(c+u')}\big)\subset \cg(c+u')$$
and for every $(\tilde\theta_0,r)\in \overline{\cg(c+\tilde u')}$, then $\big(\pi_1 \circ F^{k}(\tilde\theta_0,r)\big)_{k\in\Z_-} = (\tilde\theta_k)_{k\in \Z_-}$ is minimizing and satisfies
\begin{equation}\label{calib}
\forall k<0,\quad \tilde u(\tilde \theta_0) - \tilde u(\tilde \theta_k)  = \sum_{i=k}^{-1} S(\tilde \theta_i,\tilde\theta_{i+1}) + c(\tilde\theta_k-\tilde\theta_0) + |k|\alpha(c);
\end{equation}
\\
we will give in Appendix \ref{sspseudo} an example of a backward invariant pseudograph that doesn't correspond to any weak K.A.M. solution;

\item \label{PtAubrypseudo} moreover, if $u$ is a weak K.A.M. solution  for $\widehat T^c$, then the set $$\bigcap_{n\in\N}f^{-n}\big(\cg(c+u')\big)$$ is a { $f$-invariant} minimizing compact well-ordered set to which we can associate a unique rotation number. It results from Mather's theory that this rotation number only depends on $c$ and is equal to $\rho(c)=\alpha'(c)$; because of the convexity of $\alpha$, observe in particular that $\alpha$ is $C^1$ and $\rho$ is continuous and non-decreasing;   

\item it then follows from the first \ref{first} and the previous  \ref{same} and \ref{PtAubrypseudo} points that, as in \ref{same}, for any compact subset $K$ of $\R$, the weak K.A.M. solutions for $T^c$ with $\rho(c)\in K$ are uniformly semi-concave (i.e. for a fixed constant of semi-concavity) and then uniformly Lipschitz;

\item\label{calibprop} reciprocally, if $u$ is a weak K.A.M. solution  for $\widehat T^c$ and $(\tilde\theta_k)_{k\in \Z_-}$ verifies \eqref{calib} (we say that $(\tilde\theta_k)_{k\in \Z_-}$ calibrates $\tilde u_c$), then the sequence $(\tilde\theta_k)_{k\in \Z_-}$ is minimizing. Setting  for $k\in \mathbb Z_-$, $r_k = \frac{\partial S}{\partial \widetilde\Theta} (\tilde\theta_{k-1},\tilde\theta_k)$, the sequence $(\tilde\theta_k,r_k)_{k\in \Z_-}$ is a piece of orbit of $F$ such that $(\tilde \theta_0 , r_0) \in \pg(c+\tilde u')$ and for all $k<0$, $(\tilde \theta_k , r_k) \in \cg(c+\tilde u')$;

\item in the setting of point \ref{PtAubrypseudo}, then for every weak K.A.M. solution for $\widehat T^c$, the graph $\cg(c+\tilde u')$ contains any minimizing Aubry-Mather set with rotation number $\rho(c)$ that is minimal for the inclusion; we denote the union of these minimal  Aubry sets by $\cm^*\big(\rho(c)\big)$ (it is the Mather set). We denote  $\cm\big(\rho(c)\big)=\pi_1\Big(\cm^*\big(\rho(c)\big)\Big)$. If $\rho(c)$ is irrational,  then two possibilities may occur:
\begin{itemize}
\item either $\cm^*\big(\rho(c)\big)$ is  an invariant  Cantor set  and  $\cg(c+\tilde u')$ is contained in the unstable set of the Cantor set $\cm^*\big(\rho(c)\big)$;
\item  or $\cm^*\big(\rho(c)\big)=\cg(c+\tilde u')$ and $u$ is $C^1$.
\end{itemize}
 If $\rho(c)$ is rational, then $\cm^*\big(\rho(c)\big)$ is the union of some periodic orbits   and  $\cg(c+\tilde u')$ is contained in the union of the unstable sets of these periodic orbits.

\end{enumerate}
We noticed that to any $c\in \R$ there corresponds a unique rotation number $\rho(c)$. But it can happen that distinct numbers $c$ correspond to a same rotation number $R$. In this case, because $\rho(c)=\alpha'(c)$ is non decreasing (because of  point \ref{PtAubrypseudo}), $\rho^{-1}(R)=[c_1, c_2]$ is an interval. It can be proved that this may happen  only for rational $R$'s. 
This is a result of John Mather \cite{Mather} (where he also attributes it to Aubry) and Victor Bangert \cite{Ban2}. A simple proof can be found in \cite[Proposition 6.5]{Be3}. We will recover this fact as a byproduct of our study.

Finally, when $c$ corresponds to an irrational rotation number $\rho(c)$, then there exists only one weak K.A.M. solution up to constants.  
 The argument comes from \cite{Be3} and we will also provide a proof.

\subsection{More results on weak K.A.M. solutions}\label{ssmoreK.A.M.}

 We start by establishing that   minimizing sequences that calibrate a weak K.A.M. solution admit a rotation number\footnote{Actually, we will see in section \ref{sec4} that all minimizing sequences calibrate a weak K.A.M. solution and have a rotation number. }.
 
 \begin{lemma}\label{orderWK.A.M.}
 Let $v : \T\to \R$ be a continuous function, $c\in \R$ and $\tilde\theta_1<\tilde\theta_2$ two real numbers. Assume that $\tilde\theta'_1$ and $\tilde\theta'_2$ verify for $i\in \{1,2\}$,
 $$  \widetilde T^c \tilde v(\tilde\theta_i)=\min_{\tilde\theta'\in\R} \big(\tilde v(\tilde\theta')+S(\tilde\theta',\tilde \theta_i)+c(\tilde\theta'-\tilde\theta)\big)=\tilde v(\tilde\theta'_i)+S(\tilde\theta'_i,\tilde \theta_i)+c(\tilde\theta'_i-\tilde\theta_i),$$ 
then $\theta'_1\leqslant \theta'_2$.

Assume moreover that $v$ is semi-concave, then the previous inequality is strict.
 \end{lemma}
 
 \begin{proof}
 Let us argue by contradiction, then by Proposition \ref{PAubryfund} the following holds:
 \begin{multline*}
 \tilde v(\tilde\theta'_1)+S(\tilde\theta'_1,\tilde \theta_1)+c(\tilde\theta'_1-\tilde\theta_1)+\tilde v(\tilde\theta'_2)+S(\tilde\theta'_2,\tilde \theta_2)+c(\tilde\theta'_2-\tilde\theta_2) > \\
  > \tilde v(\tilde\theta'_2)+S(\tilde\theta'_2,\tilde \theta_1)+c(\tilde\theta'_2-\tilde\theta_1)+\tilde v(\tilde\theta'_1)+S(\tilde\theta'_1,\tilde \theta_2)+c(\tilde\theta'_1-\tilde\theta_2).
 \end{multline*}
We infer that one of the two inequalities 
$$ \tilde v(\tilde\theta'_1)+S(\tilde\theta'_1,\tilde \theta_1)+c(\tilde\theta'_1-\tilde\theta_1)> \tilde v(\tilde\theta'_2)+S(\tilde\theta'_2,\tilde \theta_1)+c(\tilde\theta'_2-\tilde\theta_1),
$$
$$\tilde v(\tilde\theta'_2)+S(\tilde\theta'_2,\tilde \theta_2)+c(\tilde\theta'_2-\tilde\theta_2)>\tilde v(\tilde\theta'_1)+S(\tilde\theta'_1,\tilde \theta_2)+c(\tilde\theta'_1-\tilde\theta_2),
$$
is valid that is a contradiction.

Let us establish the last point. If $v$ is semi-concave, by properties of the Lax-Oleinik semigroup \ref{ptdifftt}, $\tilde v$ is derivable at $\tilde\theta'_1$ and $\tilde\theta'_2$ and $\tilde\theta_i = \pi_1\circ F\big ( \tilde\theta'_i,c+\tilde v'(\tilde\theta'_i)\big)$ for $i\in \{1,2\}$ therefore $\tilde\theta'_1\neq \tilde\theta'_2$ as  $\tilde\theta_1\neq \tilde\theta_2$.

 \end{proof}

\begin{lemma}\label{Lrotnumber}
Let $u$ be a weak K.A.M. solution  for $\widehat T^c$. Let $(\theta_0,r)\in \overline{\cg(c+u')}$, and let $\tilde \theta_0\in \R$ be a lift of $\theta_0$. Let  $(\tilde\theta_k,r_k)_{k\in\Z_-}= \big(F^{k}(\tilde\theta_0,r)\big)_{k\in\Z_-}$. Then
$$\lim_{k\to -\infty} \frac{\tilde\theta_k}{k} = \rho(c).$$
\end{lemma}
\begin{proof}
Let $x_0 \in \cm\big(\rho(c)\big)$ such that $x_0\leqslant \tilde\theta_0 \leqslant x_0+1$ and $ (x_k)_{k\in\Z}$ the associated minimizing sequence. By successive applications of the previous lemma, it follows that $x_k\leqslant \tilde\theta_k \leqslant x_k+1$ for all $k\leqslant 0$. The result follows as 

$$\lim_{k\to -\infty} \frac{x_k}{k} = \rho(c).$$

%
%
\end{proof}

\begin{propos}\label{Pordreirrat}
Let $u_1$, $u_2$ be two weak K.A.M. solutions corresponding to $T^{c_1}$, $T^{c_2}$, such that $\rho(c_1)<\rho(c_2)$. Then we have
\begin{itemize}
\item $c_1<c_2$;
\item { for any $t\in \T$, if $(t,p_1)\in \partial u_1(t)$  and   $(t,p_2)\in \partial u_2(t)$ then $c_1+p_1 < c_2+p_2$; }
\item in particular, at every point of differentiability $t$ of $u_1$ and $u_2$: $c_1+u'_{1}(t)<c_2+u'_{2}(t)$.
\end{itemize}

\end{propos}
\begin{proof}
Let $\tilde u_1$ and $\tilde u_2$ be the lifts of $u_1$ and $u_2$. We introduce the notation $v(t)=\tilde u_{2}(t)-\tilde u_{1}(t) +(c_2-c_1)t$. Then $v$ is Lipschitz and thus Lebesgue everywhere differentiable and equal to a primitive of its derivative. 
Let us assume by contradiction  that  there exist $(x, c_1+p_1)\in \overline{ \cg(c_1+u_1')}$  and   $(x, c_2+p_2)\in \overline{\cg (c_2+u_2')}$
\begin{equation}\label{EineqAubry} c_2+p_2 \leqslant c_1+p_1.\end{equation}
As $\rho(c_1)\not=\rho(c_2)$,  the two graphs correspond to distinct rotation numbers. Thanks to \ref{ptinvgraph} their closures have no intersections. The inequality (\ref{EineqAubry}) is then strict. \\
We introduce the notation 
$(x^1, y^1)=(x, c_1+p_1)$ and $(x^2, y^2)=(x, c_2+p_2)$.
 Then the orbit of $(x^i, y^i)$ is denoted by $(x^i_k, y^i_k)_{k\in\Z}$. We know that the negative orbits $(x^i_k, y^i_k)_{k\in\Z_-}$, that are contained in the corresponding graphs, are minimizing. Hence the sequences $(x^i_k)_{k\in\Z_-}$ are minimizing. By Aubry's fundamental lemma, we know that they can cross at most once (hence only at $x$). But we have
\begin{itemize}
\item because of the twist condition, as $x^1_0=x^2_0$ and $y^1_0>y^2_0$, then $x^1_{-1}<x^2_{-1}$;
\item as  $\rho(c_1)<\rho(c_2)$,
and thus for $k$ close enough to $-\infty$, we have: $x^1_k>x^2_k$.
\end{itemize}
Finally we find two crossings for two minimizing sequences, a contradiction.

 We have in particular for any point $t$ of derivability of $u_1$ and $u_2$
$$c_1+u_1'(t)<c_2+u_2'(t).$$
Integrating this inequality, we deduce that $c_1<c_2$.

Finally, for any $t\in \T$, 
as for all
$(t,p_1)\in \overline{ \cg(c_1+u_1')}$  and   $(t,p_2)\in \overline{\cg (c_2+u_2')}$
\begin{equation}\label{RineqAubry} c_2+p_2 > c_1+p_1,\end{equation}
taking convex hulls, we get the result.
 
\end{proof}

In particular, we obtain the following consequence.

\begin{cor}\label{Cordre}
With the same notation as in Proposition \ref{Pordreirrat}, assume that $c_1<c_2$ are such that at least one of $\rho^{-1}(\{\rho(c_1)\})$ and $\rho^{-1}(\{\rho(c_2)\})$ is a singleton. Then
the function $t\in\R\mapsto\tilde  u_{c_2}(t)-\tilde u_{c_1}(t) +t(c_2-c_1)$ is strictly increasing.
\end{cor}
\begin{remk}
A consequence of Proposition \ref{Pordreirrat} is that the { pseudo}graphs corresponding to the weak K.A.M. solutions having distinct  rotation numbers are vertically ordered with the same order as the one between the rotation numbers.\\
\end{remk} 
Now we  recall some results that are contained in \cite{GT1} (see especially Theorem 9.3).

\begin{nota}
If  $  c\in\R$  and $n\geqslant 1$, we denote by $  \cs^c_n:\R\times \R\rightarrow \R$ the function that is defined by
$$ \cs^c_n(\tilde\theta, \widetilde\Theta)=\inf_{\substack{\tilde\theta_0=\tilde\theta \\\tilde \theta_n-\widetilde\Theta\in\Z}} \left\{\sum_{i=1}^n\big( S(\tilde\theta_{i-1},\tilde\theta_i)+c(\tilde\theta_{i-1}-\tilde\theta_i)\big)\right\}.$$
 Observe that $\cs^c_n$ is $\mathbb Z^2$-periodic.
 \end{nota}

\begin{enumerate}
\item\label{Matherinv} If $R$ is any rotation number, for any $c\in\rho^{-1}(R)$ and any weak K.A.M. solution $u$ for $\widehat T^c$, the set of invariant Borel probability measures with support in $\cg(c+u')$ is independent from $c\in\rho^{-1}(R)$ and $u$.  Those measures are called Mather measures and the union of the supports of these measures is called the {\em Mather set} for $R$ and corresponds to $\cm^*(R)$; its projection on $\T$ is denoted $\cm(R)$ whose lift to $\R$ is $\widetilde \cm(R)$;
\item We say that a function $u$ defined on a part $A$ of $\T$ is $c$-dominated if,    denoting by $\widetilde A$ the lift of $A$ to $\R$,  and $\tilde u$ a lift of $u$, we have
$$\forall \tilde\theta, \tilde\theta'\in \widetilde A,\forall n\geqslant 1, \quad  \tilde u(\tilde\theta)-\tilde u(\tilde\theta')\leqslant \cs^c_n(\tilde\theta',\tilde \theta)+ n\alpha(c);$$
a weak K.A.M. solution for $\widehat T^c$ is always $c$-dominated; {if $A=\T$ a function $u : \T\to \R$ is $c$-dominated  if and only if
$$\forall \tilde\theta, \tilde\theta'\in \R, \quad  \tilde u(\tilde\theta)-\tilde u(\tilde\theta')\leqslant S(\tilde\theta',\tilde \theta)+c(\tilde\theta'-\tilde\theta)+\alpha(c);$$}
\item\label{K.A.M.ext} If $u:\cm\big(\rho(c)\big)\rightarrow \R$ is dominated, then there exists only one extension $U$ of $u$ to $\T$ that is a weak K.A.M. solution for $\widehat T^c$.   This function is   given by
$$\forall x\in \T, \quad U(x) = \inf_{\substack{\pi(\tilde\theta)\in  \cm\textrm{$\big($}\rho(c)\textrm{$\big)$} \\\pi(\tilde\theta')=x}} \tilde u(\tilde\theta) + \cs^c(\tilde\theta,\tilde\theta'),$$
 where 
$$\cs^c(\tilde\theta,\widetilde\Theta)=\inf_{n\in\N} \big(\cs^c_n(\tilde\theta, \widetilde\Theta)+n\alpha(c)\big).$$

As we have not found it exactly stated in this way in the literature, we provide a sketch of proof for the reader's convenience in  appendix \ref{appendix-3}.
\end{enumerate}

\subsection{Proof of Theorem \ref{Tgenecont}} 
  When there is no ambiguity in the notations, we will put $\sim$ signs to signify that we consider lifts of functions defined on $\T$.
 We will need the following lemma.
\begin{lemma}\label{Llimweak} Let $(c_n)$ be a  sequence of real numbers convergent to $c$ and let $(u_{c_n})$ be a  sequence of functions uniformly convergent to $v$ such that $u_{c_n}$ is  a weak K.A.M. solution for $\widehat T^{c_n}$. Then {\color{blue} $v$ } is a weak K.A.M. solution for $\widehat T^c$.
\end{lemma}
\begin{proof}
We know from Equation (\ref{Esolweakk}) that 
$$\tilde u_{c_n}(\tilde\theta)=\inf_{\tilde\theta'\in\R} \big(\tilde u_{c_n}(\tilde\theta')+S(\tilde\theta', \tilde\theta)+c_n(\tilde\theta'-\tilde\theta)+\alpha(c_n)\big).$$
Because of the superlinearity of $S$ and the fact that the $u_{c_n}$ and $c_n$ are uniformly bounded, there exists a fixed compact set $I$ in $\R$ such that for every $n$, we have
$$\tilde u_{c_n}(\tilde\theta)=\inf_{\tilde\theta'\in I} \big(\tilde u_{c_n}(\tilde\theta')+S(\tilde\theta',\tilde \theta)+c_n(\tilde\theta'-\tilde\theta)+\alpha(c_n)\big).$$
We  deduce from the uniform convergence of $(u_{c_n})$ to $v$ that
$$\tilde v(\tilde\theta)=\inf_{\tilde\theta'\in I} \big(\tilde v(\tilde\theta')+S(\tilde\theta', \tilde\theta)+c(\tilde\theta'-\tilde\theta)+\alpha(c)\big).$$
As we could do the same proof for $I$ as large as wanted, we have in fact
\begin{equation}\label{ElimwehK.A.M.} \tilde v(\tilde\theta)=\inf_{\tilde\theta'\in \R} \big(\tilde v(\tilde\theta')+S(\tilde\theta', \tilde\theta)+c(\tilde\theta'-\tilde\theta)+\alpha(c)\big).\end{equation}

\end{proof}
Let us now prove Theorem  \ref{Tgenecont}. We will start with a fundamental uniqueness property of weak K.A.M. solutions for a wide class of cohomology classes.\\

\begin{propos}\label{unicite}
Let $R\subset \R$ be a real number and set $[a_1,a_2] = \rho^{-1}(\{R\})$. Then, up to constants, there exists a unique weak K.A.M. solution for $\widehat T^{a_1}$ (resp. $\widehat T^{a_2}$).
\end{propos}

\begin{proof}
Let us prove the result for $a_2$, the proof being similar for $a_1$. Let $(c_n)_{n\in \N}$ be a decreasing sequence of real numbers converging to $a_2$, such that $\big(\rho(c_n)\big)_{n\in \N}$ is decreasing (and converges to $R$). For all $n\in \N$, let $u_n : \T \to \R$ be a weak K.A.M. solution at cohomology $c_n$ such that $u_n(0) = 0$.

Then by Proposition \ref{Pordreirrat}, $(c_n+u'_{n})_{n\in\N}$ is a decreasing sequence and then $\big(\tilde v_n:\tilde\theta\in[0, 1]\mapsto c_n\tilde\theta+\tilde u_{n}(\tilde\theta)\big)_{n\in\N}$ is also a decreasing sequence, thus convergent and even uniformly convergent by the Ascoli Theorem. By Lemma \ref{Llimweak}, $\displaystyle{ \tilde u_{a_2}(\tilde\theta)=\lim_{n\rightarrow \infty} \tilde v_n(\tilde\theta)-c_n\tilde\theta}$ defines a weak K.A.M. solution for $T^{a_2}$ such that $u_{a_2}(0)=0$ and $u_{a_2}' = \lim\limits_{n\rightarrow \infty} u'_n$ almost everywhere.

Let us assume that $v$ is another weak K.A.M. solution for $\widehat T^{a_2}$ that vanishes at $0$. Because of Proposition \ref{Pordreirrat}, we have for all $n\in \N$,
$$ c_n+u_n'>a_2+v'.$$
Taking the limit in these inequalities  and using the definition of $u_{a_2}$, we deduce that $v'\leqslant u'_{a_2}$. As $0=\int_\T v'=\int_\T u'_{a_2}$, we deduce that $u'_{a_2}=v'$ Lebesgue almost everywhere and then $v=u_{a_2}$.

\end{proof}

\begin{nota}
We use the notation  $\ci\subset \R$ is the set of $ c\in \R$ such that  $\rho^{-1}(\{ \rho(c)\}) = \{c\}$. This is the set of cohomology classes where $\rho$ is strictly increasing\footnote{Let us remind the reader that $\ci$ contains $\R \setminus \Q$. But we will not use this fact.}.
\end{nota}

It is easily verified that the closure $ \cii$ consists in the union of all the extremities $\{a_1,a_2\}$ of intervals $[a_1,a_2] = \rho^{-1}(\{R\})$ for $R\in \R$. This justifies the next:

\begin{nota}
When $c\in \cii$, we will denote by $u_c$ the  (unique) solution such that $u_c(0)=0$.\end{nota}

Let us prove that any extension $c\mapsto u_c$ that maps $c$ on a weak K.A.M. solution for $\widehat T^c$ that vanishes at $0$ is continuous at every $c\in \cii$. Let us consider a  sequence $(c_n)_{n\in \N}$  that converges to $c\in\cii$.  Then the sequence  $(u_{c_n})_{n\in \N}$ is made of equi semi-concave and then equiLipschitz functions. As all functions vanish at $0$, the sequence is also equi-bounded.
 Because of the Ascoli Theorem it is relatively compact for the uniform convergence.
Because of Lemma \ref{Llimweak}, all its accumulation points are weak K.A.M. solutions for $\widehat T^c$ that vanish at $0$. It follows that the sequence uniformly converges to the unique such function $u_c$.\\
This gives the wanted continuity at every point of $\cii$.\\

 Building  a function $u$, the only problem of continuity we have now to consider is at the points of the set $\R\setminus \cii$.\\
 Observe that if we find a continuous extension to $\T\times \R$ such that every $u_c$ is a weak K.A.M. solution for $\widehat T^c$, replacing $u_c$ by $u_c-u_c(0)$, we obtain an extension as wanted.

Let us now assume that $R$ is a real number such that $\rho^{-1}(\{R\})=[a_1,a_2]$, with $a_1<a_2$\footnote{Again, $R$ is necessarily rational but we do not need this fact.}.
Because $u_{a_1}$ and $u_{a_2}$ are weak K.A.M. solutions, they are dominated and we have
$$\forall x, y\in \R, \forall n\geqslant 1, \quad \tilde u_{a_i}(x)-\tilde u_{a_i}(y)\leqslant  \cs^{a_i}_n(x, y)+n\alpha (a_i).$$
Let $c=\lambda a_1+(1-\lambda) a_2\in [a_1, a_2]$. We use the notation $v_c=\lambda u_{a_1}+(1-\lambda) u_{a_2}$. Observe that $\alpha (c)=\lambda\alpha(a_1)+(1-\lambda)\alpha(a_2)$ because $\alpha'=R$ is constant on $[a_1, a_2]$. Then we have
$$\forall x, y\in \widetilde\cm\big(R\big),\quad \tilde v_c(y)- \tilde v_c(x)\leqslant \cs^c_n(x,y)+n\alpha(c);$$
i.e. $v_c$ is $c$-dominated on $\cm\big(R\big)$. We deduce from Point (\ref{K.A.M.ext}) of subsection \ref{ssmoreK.A.M.} that there exists only one extension   $u_c$ of  $v_{c}$ restricted to $\cm\big(R\big)$ that is a weak K.A.M. solution for $\widehat T^c$.\\

Let us  prove that $c\in [a_1, a_2]\mapsto u_c$ is continuous. By definition of $u_c$, the map $c\mapsto u_{c|\cm(R)}$ is continuous. Let us now consider a sequence $(c_n)$ in $[a_1, a_2]$ that converges to some $c\in [a_1, a_2]$. By the Ascoli Theorem the set $\{ u_{c_n}, n\in\N\}$ is relatively compact for the uniform convergence distance. Let $U$ be a limit point of the sequence $(u_{c_n})$. By Lemma \ref{Llimweak}, we know that $U$ is a weak K.A.M. solution for $\widehat T^c$. Moreover, we have $U_{|\cm(R)}=u_{c|\cm(R)}$.  Using Point (\ref{K.A.M.ext}) of subsection \ref{ssmoreK.A.M.}, we deduce that $U=u_c$ and the wanted continuity.

%
%

In appendix A of \cite{Be2}, it is proved that the uniform  convergence of a sequence of equi-semi-concave functions implies their convergence $C^1$ in some sense. This implies for the function $u$ given in Theorem  that  if $c_n\rightarrow c$, if $\theta_n \rightarrow \theta$ and if $u_{c_n}$ is derivable at $\theta_n$ and $u_c$ at $\theta$, we have
 $$\lim_{n\rightarrow \infty} \frac{\partial u}{\partial \theta}(\theta_n, c_n)=\frac{\partial u}{\partial \theta}(\theta, c),$$ i.e. that  the map $(\theta, c)\mapsto \frac{\partial u_c}{\partial \theta}(\theta)$ is continuous.\\


  We end this section with the proof of points \eqref{croissant} and \eqref{u-lip} of Theorem \ref{Tgenecont}. Let us state a lemma:

\begin{lemma}
Let  $c_1<c_2$  be two real numbers. Let  $v_1,v_2 : \T\to \R$ be continuous   functions.

If the function $\theta \mapsto (\tilde v_2-\tilde v_1)(\tilde\theta) + (c_2-c_1)\tilde\theta$ is non-decreasing, then so is the function $\tilde\theta \mapsto (\widetilde{ T}^{c_2}\tilde v_2-\widetilde T^{c_1}\tilde v_1)(\tilde\theta) + (c_2-c_1)\tilde\theta$.
\end{lemma}

\begin{proof}
Let $\tilde\theta < \tilde\theta'$ be two real numbers. By definition of the operators $\widetilde T_{c_i}$ there exist $\tilde\theta_2'$ and $\tilde\theta_1$ such that
$$\widetilde T^{c_2}\tilde v_2(\tilde\theta') = \tilde v_2(\tilde\theta'_2) + S(\tilde\theta'_2,\tilde\theta') + c_2(\tilde\theta'_2-\tilde\theta'),$$
$$\widetilde T^{c_1}\tilde v_1(\tilde\theta) = \tilde v_1(\tilde\theta_1) + S(\tilde\theta_1,\tilde\theta) + c_1(\tilde\theta_1-\tilde\theta).$$
There are two cases to consider:
\begin{itemize}
\item if $\tilde\theta_2' <\tilde \theta_1$ we use Aubry's fundamental lemma to obtain
\begin{align*}
\widetilde T^{c_2}\tilde v_2(\tilde\theta') +\widetilde T^{c_1}\tilde v_1(\tilde\theta) &= \tilde v_2(\tilde\theta'_2) + S(\tilde\theta'_2,\tilde\theta') + c_2(\tilde\theta'_2-\tilde\theta')+ \tilde v_1(\tilde\theta_1) + S(\tilde\theta_1,\tilde\theta) + c_1(\tilde\theta_1-\tilde\theta) \\
&> \tilde v_2(\tilde\theta'_2) + S(\tilde\theta'_2,\tilde\theta) + c_2(\tilde\theta'_2-\tilde\theta')+ \tilde v_1(\tilde\theta_1) + S(\tilde\theta_1,\tilde\theta') + c_1(\tilde\theta_1-\tilde\theta) \\
& \geqslant  \widetilde T^{c_2}\tilde v_2(\tilde\theta) +\widetilde T^{c_1}\tilde v_1(\tilde\theta')+ (c_2-c_1) (\tilde\theta-\tilde\theta').
\end{align*}
After rearranging the terms, this reads
$$ \widetilde T^{c_2}\tilde v_2(\tilde\theta')-\widetilde T^{c_1}\tilde v_1(\tilde\theta') +(c_2-c_1)\tilde\theta' >  \widetilde T^{c_2}\tilde v_2(\tilde\theta)-\widetilde T^{c_1}\tilde v_1(\tilde\theta) +(c_2-c_1)\tilde\theta.$$

\item if $\tilde\theta_2' \geqslant\tilde \theta_1$ we use the hypothesis on $\tilde\theta \mapsto (\tilde v_2-\tilde v_1)(\tilde\theta) + (c_2-c_1)\tilde\theta$ to show that 
$ \tilde v_2(\tilde\theta'_2) + \tilde v_1(\tilde\theta_1) \geqslant \tilde v_2(\tilde\theta_1) + \tilde v_1(\tilde\theta'_2)+ (c_2-c_1) (\tilde\theta_1-\tilde\theta'_2) $ and then
\begin{align*}
\widetilde T^{c_2}\tilde v_2(\tilde\theta') +\widetilde T^{c_1}\tilde v_1(\tilde\theta) &= \tilde v_2(\tilde\theta'_2) + S(\tilde\theta'_2,\tilde\theta') + c_2(\tilde\theta'_2-\tilde\theta')+ \tilde v_1(\tilde\theta_1) + S(\tilde\theta_1,\tilde\theta) + c_1(\tilde\theta_1-\tilde\theta) \\
&\geqslant  \tilde v_2(\tilde\theta_1) + S(\tilde\theta'_2,\tilde\theta') + c_2(\tilde\theta_1-\tilde\theta')+ \tilde v_1(\tilde\theta'_2) + S(\tilde\theta_1,\tilde\theta) + c_1(\tilde\theta'_2-\tilde\theta) \\
&\geqslant     \widetilde T^{c_2}\tilde v_2(\tilde\theta) +\widetilde T^{c_1}\tilde v_1(\tilde\theta')+ (c_2-c_1) (\tilde\theta-\tilde\theta').
\end{align*}
As before, this gives the result after rearranging terms.
\end{itemize}

\end{proof}
Let us now conclude that the function $u$ constructed verifies the requirements of \eqref{croissant} and \eqref{u-lip}. Let $R$ be  a real number and let us, as previously, introduce the notations $\rho^{-1}(R)=[a_1,a_2]$. As seen before, we denote by $u_{a_1}$ and $u_{a_2}$ the unique weak K.A.M. solutions for $T^{a_1}$ and $T^{a_2}$ vanishing at $0$. We have proven that $\tilde\theta \mapsto (\tilde u_{a_2}-\tilde u_{a_1})(\tilde\theta) + (a_2-a_1)\tilde\theta$ is non-decreasing.

Let $c=\lambda a_1+(1-\lambda) a_2\in [a_1, a_2]$. We use again the notation $v_c=\lambda u_{a_1}+(1-\lambda) u_{a_2}$ and recall that $\alpha (c)=\lambda\alpha(a_1)+(1-\lambda)\alpha(a_2)$ because $\alpha'=R$ is constant on $[a_1, a_2]$. It follows that $\tilde v_c$ is $c$ dominated and that if $a_1\leqslant c<c' \leqslant a_2$, the function $\tilde\theta \mapsto (\tilde v_{c'}-\tilde v_c)(\tilde\theta) + (c'-c)\tilde\theta$ is non decreasing.  

Finally, as $v_c$ is $c$-dominated, it can be proved that the function $u_c$ constructed verifies
$$\forall \tilde\theta\in \R,\quad \tilde u_c(\tilde\theta) = \lim_{n\to +\infty} (\widetilde T^c)^n \tilde v_c (\tilde\theta)+n\alpha(c),$$
the limit being that of an increasing sequence. Hence the fact that $\tilde\theta \mapsto (\tilde u_{c'}-\tilde u_c)(\tilde\theta) + (c'-c)\tilde\theta$ is non decreasing follows from successive applications of the previous lemma.

To prove \eqref{u-lip}, if $c'\leqslant c $ and $\tilde\theta\in [0,1]$ then
$$0= (\tilde u_{c'}-\tilde u_c)(0) \leqslant (\tilde u_{c'}-\tilde u_c)(\tilde\theta) + (c'-c)\tilde\theta\leqslant (\tilde u_{c'}-\tilde u_c)(1) + (c'-c) = c'-c.$$
It follows that 
$$(c-c')\tilde\theta \leqslant (\tilde u_{c'}-\tilde u_c)(\tilde\theta) \leqslant (c'-c)(1-\tilde\theta).$$
Hence $\tilde u$ is uniformly $1$-Lipschitz in $c$ and the result follows.

\subsection{More on the constructed function: proof of Theorem \ref{Tpseudofol}}

In this paragraph, $u : \A \to \R$ is any function given by Theorem \ref{Tgenecont} meaning that 
\begin{itemize}
\item  $u$ is continuous;
\item $u(0,c)=0$;
\item each $u_c= u(\cdot ,c)$ is a weak K.A.M. solution for the cohomology class $c$. 
\end{itemize}
We aim to give the range of the 
 map 
$(\theta,c) \mapsto \big(\theta, c+\frac{\partial u}{\partial \theta}(\theta,c)\big)$.  The following proposition asserts that any ESTwD  is weakly integrable in the sense that $\A$ is covered by Lipschitz circles arising from weak K.A.M. solutions.

Recall that  $\mathcal{PG}(c+u'_c) =\{ (0,c)+\partial u_c (t), \quad t\in \T\}$ is the full pseudograph of $c+u'_c$.

\begin{propos}\label{+}
The following holds:
\begin{equation}\label{a0}
 \bigcup_{c\in \R} \mathcal{PG} (c+u'_c)=\bigcup_{\substack{t\in \T \\ c\in \R}} (0,c)+\partial u_c (t) = \A.
\end{equation}
\end{propos}
Let us define two auxiliary functions  with values in $\R\cup\{ +\infty, -\infty\}$:
$$ \forall \theta \in \T,\quad \eta_+(\theta) = \sup
 \big\{p\in \R ; \quad \exists c\in\R; \quad (\theta,p)\in \overline{\cg (c+u_c')}\big\},
$$
and 
$$ \forall \theta \in \T,\quad \eta_-(\theta) = \inf
 \big\{p\in \R ; \quad  \exists c\in\R; \quad (\theta,p)\in \overline{\cg (c+u_c')}\big\}.
$$
Finally define $\A_0 = \big\{(\theta,c)\in \A , \quad \eta_-(\theta)<c<\eta_+(\theta)\big\}$.

The following lemma is proved in Appendix \ref{ssLfullman}.

\begin{lemma}\label{Lfullman}
For all $c\in \R$, $\mathcal{PG}(c+u'_c)$ is a Lipschitz one dimensional compact manifold, hence it is an essential circle. 
\end{lemma}

It follows that the set $\A_0$ is open and connected (we will see at the end that it is in fact $\A$).
Indeed, by Jordan's theorem and Proposition \ref{Pordreirrat}, for $c<c'$ such that $\rho(c)<\rho(c')$, the set $\big\{(t,p)\in \A , \quad c+\partial u_c(t)<p<c'+\partial u_{c'} (t)\big\}$ is open and connected. Now $\A_0$ is an increasing union of such sets. 
 
\begin{propos}
The following equality holds:
$$\A_0 = \bigcup_{c\in \R} \mathcal{PG} (c+u'_c).$$
\end{propos}

\begin{proof}
 We   denote by $\mathcal B= \bigcup_{c\in \R} \mathcal{PG} (c+u'_c)$. Observe that $\cb\subset \A_0$.

First we prove that $\mathcal B$ is closed in $\A_0$.   Let $(t_n,p_n)\in \mathcal{PG}(c_n+u'_{c_n})$ be a sequence converging to $(t,p)\in \A_0$. By definition of $\A_0$, there are $C_0<C_1$ and $(t,P_0) \in \mathcal{PG}(C_0+u'_{C_0})$,  $(t,P_1) \in \mathcal{PG}(C_1+u'_{C_1})$ such that such that $P_0<p<P_1$. 
Now let $c_-<C_0<C_1<c_+$ be such that $\rho(c_-)$, $\rho(c_+)$ are irrational and $$\rho(c_-)<\rho(C_0)<\rho(C_1)<\rho(c_+).$$
  As the pseudographs are vertically ordered (Proposition \ref{Pordreirrat}), $(t,p)$ is trapped in the open sub-annulus between $\mathcal{PG}(c_-+u'_{c_-})$ and $\mathcal{PG}(c_++u'_{c_+})$. It follows that for $n$ large enough, so is $(t_n,p_n)$.   Hence $\mathcal{PG}(c_n+u'_{c_n})$ is a full pseudograph that contains a point strictly between  $\mathcal{PG}(c_-+u'_{c_-})$and $\mathcal{PG}(c_++u'_{c_+})$. Proposition \ref{Pordreirrat} implies that  $\rho(c_-)\leqslant\rho(c_n)\leqslant \rho (c_+)$. As $\rho(c_-)$, $\rho(c_+)$ are irrational, there is a unique weak K.A.M. solution for these rotation numbers and then $\rho(c_n)\notin \{ \rho(c_-), \rho (c_+)\}$. We deduce that
$\rho(c_-)<\rho(c_n)< \rho (c_+)$  and then that $c_-<c_n<c_+$.

Up to extracting, we may assume that $c_n \to c_\infty$ and by continuity of the pseudographs  with respect to $c$ (for the Hausdorff distance, see Proposition \ref{hausdorff}), it follows that $(t,p)\in \mathcal{PG}(c_\infty+u'_{c_\infty})\subset \A_0$.

Next we prove that $\mathcal B=\A_0$. We argue by contradiction, by the first part, if this is not the case, there is an open ball $B=(\theta_0,\theta_1)\times (r_0,r_1)$ such that $\overline B\subset \A_0 \setminus \mathcal B$. 

We say that a topological essential circle $\mathcal C$ is above $B$ if $B$ is included in the lower connected component of $\A \setminus \mathcal C$\footnote{Recall that by Jordan's theorem, $\A \setminus \mathcal C$ has two open connected components, one we call upper that contains $\T \times (k,+\infty)$  and one  we call lower, that contains $\T \times (-\infty, -k)$ for $k$ large enough.} and $\mathcal C$ is under $B$ if $B$ is included in the upper connected component of $\A \setminus \mathcal C$. Therefore, if we set $EC_B$ the set of essential circles   $\mathcal C \subset \A\setminus B$, $EC_B$ is the union of circles above $B$:  $EC_B^+$ and those under $B$: $EC_B^-$.

We will prove that 
\begin{lemma}
Both $EC_B^+$ and $EC_B^-$ are open subsets of $EC_B$ for the Hausdorff distance.
\end{lemma}
\begin{proof}
We prove it for $EC_B^+$. Let $\mathcal C^+$ be a circle above $B$. As the lower connected component of $\A \setminus \mathcal C^+$ is path connected, there is a continuous  path $\gamma : [0,+\infty) \to \A \backslash \cC^+$ such that $\gamma(0)\in B$ and $\gamma(t) = (0,-t)$ for all the $t$ large enough. Let $\varepsilon>0$ be such that  $\mathcal C^+$ is at distance greater than $\varepsilon$ from $\gamma$. If $\mathcal C^-$ is any circle under $B$, then it must intersect $\gamma$. Hence $d_H(\mathcal C^-,\mathcal C^+)>\varepsilon$ where $d_H$ stands for the Hausdorff distance. This proves the lemma. 
\end{proof}
We will obtain a contradiction as $\R$ is connected and the map $c\mapsto \mathcal{PG}(c+u'_c)$ is continuous for the Hausdorff distance, provided  we prove that for $c$ large, $\mathcal{PG}(c+u'_c)$ is above $B$ while for $c$ small $\mathcal{PG}(c+u'_c)$ is under $B$.

\begin{lemma}
For $c$ large, $\mathcal{PG}(c+u'_c)$ is above $B$ while for $c$ small $\mathcal{PG}(c+u'_c)$ is under $B$.
\end{lemma}
\begin{proof}
We establish only the first fact. Let $\theta_*\in (\theta_0,\theta_1)$. By definition of $\eta_+$, there exists $C$ such that 
$$\forall c>C, \ \forall (\theta_*, p)\in  \mathcal{PG} (c+u_c'), \quad p>r_1.$$

It follows that for $t>0$ small enough, we have
$$\forall (\theta_*,p) \in \varphi_{-t}\big(\mathcal{PG} (c+u_c')\big), \quad  p>r_1$$
 where $\varphi$ denotes here the flow of the pendulum and
 
 $$\varphi_{-t}\big(\mathcal{PG} (c+u_c')\big)\cap B=\varnothing.$$
 
  But it is proved in \cite{Arna2} that $\varphi_{-t}\big(\mathcal{PG} (c+u_c')\big)$ is the Lipschitz graph  of a function $\alpha_t : \T \to \R$ for small $t>0$. Hence it follows from the intermediate value theorem that $\alpha(\theta) >r_1$ for $\theta \in (\theta_0, \theta_1)$ and it becomes obvious that  $B = (\theta_0,\theta_1)\times (r_0,r_1)$ is under $\varphi_{-t}\big(\mathcal{PG} (c+u_c')\big)$. Letting $t\to 0$ and passing to the limit, we obtain that $B$ is under $\mathcal{PG} (c+u_c')$.
\end{proof}

\end{proof}

In order to conclude, we have to prove that $\A = \A_0$ which is equivalent to proving that $\eta_+$ is identically $+\infty $ and $\eta_-$ is identically $-\infty$. We will establish the result for $u_+$.

\begin{lemma}\label{Lhorscomp}
Let $[a,b]$ be a segment, there exists $C >0$ depending on 
 $[a, b]$ such that if $|c|>C$ then 
$$ \forall \theta\in [0,1], \theta' \in [a,b], \quad  S(\theta,\theta') + c(\theta-\theta') > \min_{n\in \Z}S(\theta,\theta'+n) + c(\theta-\theta'-n).$$

\end{lemma}
\begin{proof}
Let us set $\Delta = \max \left\{ \Big|\frac{\partial S}{\partial \theta'}(\theta,\theta')\Big| , \ \ \theta\in [0,1], \theta' \in [a-1, b+1]\right\}$ and $C=\Delta +1$.

If $|c|>C$ two cases may occur:
\begin{itemize}
\item either $c>\Delta+1$. In this case, if $(\theta,\theta')\in  [0,1]\times [a,b]$, by Taylor's inequality we find 
$$S(\theta,\theta')+ c(\theta-\theta') >S(\theta,\theta')+ c\big(\theta-(\theta'+1)\big)+\Delta \geqslant  S(\theta,\theta'+1)+ c\big(\theta-(\theta'+1)\big);$$
\item or $c <-\Delta -1$, in which case 
$$S(\theta,\theta')+ c(\theta-\theta') >S(\theta,\theta')+ c\big(\theta-(\theta'-1)\big)+\Delta \geqslant  S(\theta,\theta'-1)+ c\big(\theta-(\theta'-1)\big).$$

\end{itemize}

\end{proof}

\begin{cor}
The function $\eta_+$ is identically $+\infty$.
\end{cor}
\begin{proof}
Let  us fix $A>0$. We assume that for all $(\theta,p)\in  \mathcal{PG}(u'_0)$, then $|p|\leqslant A$ (or in other words, $u_0$ is $A$-Lipschitz).  As every map $\theta\mapsto \frac{\partial S}{\partial \Theta}(\theta, \Theta_0)$ is a decreasing diffeomorphism of $\R$, there exists a constant $B>0$ such that for every $\Theta_0\in [0, 1]$, we have
$$\theta>B\Rightarrow  \frac{\partial S}{\partial \Theta}(\theta, \Theta_0)<-(A+1)\quad{\rm and}\quad \theta<-B\Rightarrow \frac{\partial S}{\partial \Theta}(\theta, \Theta_0)>A+1.$$
Let $C$ be the constant given by Lemma \ref{Lhorscomp} for the segment $[-B,B]$  and let us choose   $c>\sup\{ B, C\}$.    Let $\tilde\theta_0\in [0, 1]$ be any derivability point of $u_c$.   Because of Lemma \ref{Lhorscomp}, if $\tilde u_c$ is a lift of $u_c$ and if $\tilde\theta_{-1}$ verifies 
$$\tilde u_c (\tilde\theta_0) = \inf_{\tilde\theta\in \R} \tilde u_c(\tilde\theta) + S(\tilde\theta,\tilde\theta_0) +c(\tilde\theta-\tilde\theta_0)= \tilde  u_c(\tilde\theta_{-1}) + S(\tilde\theta_{-1},\tilde\theta_0) + c(\tilde\theta-\tilde\theta_0),$$
then $\tilde \theta_{-1} \notin [-B,B]$ and then $\Big|\frac{\partial S}{\partial \Theta}(\tilde\theta_{-1}, \tilde\theta_0)\Big|>A+1$.

 We deduce from point \ref{ptdifftt} of section \ref{ssweakk} that $F\big(\tilde\theta_{-1}, c+\tilde u_c'(\tilde \theta_{-1})\big)=\big(\tilde \theta_0, c+\tilde u_c'(\tilde\theta_0)\big)$ and then
 $$c+\tilde u_c'(\tilde\theta_0) =   \frac{\partial S}{\partial \Theta} (\tilde\theta_{-1},\tilde\theta_0),$$
and then $|c+\tilde u_c'(\theta_0)| >A+1$.\\
As  $\int_0^1 \big(c+\tilde u'_c(s)\big) ds = c >0$, we can choose $\theta_0$ such that  $c+\tilde u'_c(\theta_0) >0$ and so $c+\tilde u_c'(\theta_0)>A+1$. 

As the pseudographs are vertically ordered (Proposition \ref{Pordreirrat}), $\mathcal{PG}(c+u'_c)$ is above $\mathcal{PG}(u'_0)$. We conclude that for all derivability point $\theta$ of $u_c $ then $c+\tilde u_c'(\tilde\theta) >A+1$. Finally, the whole full pseudograph $\mathcal{PG}(c+u'_c)$ lies above the circle $\{(t,A) , \ \ t\in \T\}$.



We have just established that if $c> B$, then $\mathcal{PG}(c+u'_c)$ lies above the circle $\{(t,A) , \ \ t\in \T\}$, that concludes the proof.

\end{proof}

Using technics given in \cite{Arna4},  we will prove in Proposition \ref{Pconfullps} of Appendix  \ref{AppB3} that  the map that maps $c$ on the   full pseudograph\footnote{see the definition in subsection \ref{ssweakk}}  $\mathcal{PG}(c+  u_c') =\{ (0,c)+\partial  u_c (t), \quad t\in \T\}$ of $c+u_c'$ is continuous for the Hausdorff distance.

 Point (3) of Theorem \ref{Tpseudofol} is a result of Proposition \ref{Pordreirrat}.

\vspace{5mm}

\section{Proof of the implication (\ref{pt2Tgeneder}) $\Rightarrow$  (\ref{pt1Tgeneder})  in   Theorem \ref{Tgeneder}}\label{secsecond}
We use the same notations as in Theorem \ref{Tgenecont}. We assume that the map $u$ is   $C^1$. \\
Then the graph of every $\eta_c=c+\frac{\partial u_c}{\partial \theta}$ is a  continuous graph that is backward invariant, hence invariant.
If for $c_1\not=c_2$ the two graphs of $\eta_{c_1}$ and $\eta_{c_2}$ have a non-empty intersection, then their common rotation number is rational because an ESTwD  has at most one invariant curve with a fixed irrational rotation number (see \cite{He1} or Theorem \ref{min-irr} below).   Moreover, for every $c\in [c_1, c_2]$, we have $\rho(c)=\rho(c_1)$.\\
 Using results   of \cite{Ban} (see section 5),  we know that above any $\theta\in \T$, there are at most two $r_1, r_2\in\R$ such that the orbit of $(\theta, r_i)$ is minimizing with rotation number $\rho(c_1)$. As $c_1\not=c_2$, there exists then $\theta\in \T$ such that $r_1=\eta_{c_1}(\theta)\not= \eta_{c_2}(\theta)=r_2$. But for $c\in [c_1, c_2]$, the orbit of $\big(\theta, \eta_c(\theta)\big)$ is minimizing with rotation number equal to $\rho(c_1)$ and then $\eta_c(\theta)\in \{ r_1, r_2\}$. As $c\mapsto \eta_c(\theta)$ is continuous with values in $\{\eta_{c_1}(\theta), \eta_{c_2}(\theta)\}$ and satisfies $\eta_{c_1}(\theta)\not=\eta_{c_2}(\theta)$, we obtain a contradiction.\\
 
   So finally the graphs of the $\eta_c$ define a lamination of $\A$ and then $f$ is $C^0$-integrable.

\section{Properties of infinite minimizing sequences and weak K.A.M. solutions}\label{sec4}

In this section we study properties of sequences $(\tilde\theta_i)_{i\in \Z_-}\in \R^{\Z_-}$ that are minimizing. By symmetry, similar statements hold for sequences   $(\tilde\theta_i)_{i\in \Z_+}$. We start with the following improvement of Lemma \ref{Lrotnumber}  that proves the beginning of Theorem \ref{Ttwistvertical}.

\begin{propos}\label{Protnumber}
Let $(\tilde\theta_i)_{i\in \Z_-}\in \R^{\Z_-}$ be a minimizing sequence, then the limit 
$\lim\limits_{k\to -\infty} \frac{\tilde\theta_k}{k} $ exists.
\end{propos}

\begin{proof}
Let us set for $i\in \mathbb Z_-$, $r_i = \frac{\partial S}{\partial \widetilde\Theta} (\tilde\theta_{i-1},\tilde\theta_i)$ so that $(\tilde\theta_i,r_i)_{i\in \mathbb Z_-}$ is a piece of orbit of $F$. By Theorem \ref{Tpseudofol}, there exists $c\in \R$ and a weak K.A.M. solution $\tilde u_c$, for $\widetilde T^c$, such that $(\tilde\theta_0,r_0)\in \mathcal{PG} (c+\tilde u'_c)$. Let $\{\tilde\theta_0\} \times [p_0,p'_0] = \partial \tilde u_c ( \tilde \theta_0) $ and for $k\in \mathbb Z_-$, let us define $x_k = \pi_1\circ F^k ( \tilde \theta_0,c+p_0)$ and $x'_k = \pi_1\circ F^k ( \tilde \theta_0,c+p'_0)$. As $(\tilde \theta_0,c+p_0), (\tilde \theta_0,c+p_0') \in \overline{ \mathcal{G} (c+\tilde u'_c)}$, the sequences $(x_k)_{k\in \mathbb Z_-} $ and $(x'_k)_{k\in \mathbb Z_-} $ are minimizing and 
$$\lim_{k\to -\infty} \frac{x_k}{k} = \lim_{k\to -\infty} \frac{x'_k}{k} = \rho(c)  $$
thanks to Lemma \ref{Lrotnumber}.  Note that if $c+p_0 =r_0 $ or if $c+p_0' = r_0$ the result holds.

In the other cases, as $c+p_0 < r_0 < c+p_0'$, it follows from the twist condition that $ x_{-1}> \tilde \theta_{-1} > x'_{-1}$. By Aubry's fundamental lemma (Proposition \ref{PAubryfund}), we infer that for all $k<0$, $ x_{k}> \tilde \theta_{k} > x'_{k}$ and the result follows.
\end{proof}

 Before continuing our study, we need to introduce some notations. In the rest of this section, $(\tilde\theta_i)_{i\in \Z_-}\in \R^{\Z_-}$ will be a minimizing sequence and associated to it, $r_i = \frac{\partial S}{\partial \widetilde\Theta} (\tilde\theta_{i-1},\tilde\theta_i)$ so that $(\tilde\theta_i,r_i)_{i\in \mathbb Z_-}$ is a piece of orbit of $F$. We will set $\rho_0$ the limit given by Proposition \ref{Protnumber}.
 
 We anticipate on a Theorem (that implies Theorem \ref{solK.A.M.calib}) that will be proved later in two steps by distinguishing wether $\rho_0$ is rational or not:
 \begin{theorem}\label{caliball}
Let  $(\tilde\theta_i)_{i\in \Z_-}\in \R^{\Z_-}$ be a minimizing sequence. There exists a cohomology class $c\in \R$ and a weak K.A.M. solution $\tilde u_c$ for $\widetilde T^c$ such that $(\tilde\theta_i)_{i\in \Z_-}\in \R^{\Z_-}$ calibrates $\tilde u_c$.
 \end{theorem}
We can already deduce some properties of minimizing sequences reminiscent of orbits of circle homeomorphisms. 

\begin{cor}\label{permin}
Let $(\tilde\theta_i)_{i\in \Z_-}\in \R^{\Z_-}$ be a minimizing sequence and $\rho_0\in \R$ its rotation number given by Proposition \ref{Protnumber}. If there exist $p$ and $q $  integers with $q<0$ and $\tilde \theta_q = \tilde \theta_0+p$ then $(\tilde\theta_i+p)_{i\in \Z_-} = (\tilde\theta_{i+q})_{i\in \Z_-}\in \R^{\Z_-}$. In this case, $(\tilde\theta_i)_{i\in \Z_-}\subset \widetilde {\mathcal{M}}(p/q)$.
\end{cor}
\begin{proof}
Let $c\in \R$ and $\tilde u_c$  be a weak K.A.M. solution given by Theorem \ref{caliball} such that $(\tilde\theta_i)_{i\in \Z_-}\in \R^{\Z_-}$ calibrates $\tilde u_c$. We know from \ref{calibprop}, recalling properties of weak K.A.M. solutions, that  $(\tilde\theta_0,r_0) \in \pg(c+\tilde u'_c)$  and $(\tilde\theta_k,r_k) \in \cg(c+\tilde u'_c)$ for all $k<0$. In particular, by periodicity,   $\tilde u_c$ is derivable at $\tilde \theta_0$ and $r_0= c+\tilde u'_c (\tilde \theta_0) =  c+\tilde u'_c (\tilde \theta_q) = r_q$. Then we obtain that 
$$\tilde \theta_{k+q} = \pi_1\circ F^k(\tilde\theta_q, r_q) =  \pi_1\circ F^k(\tilde\theta_0, r_0)+p = \tilde\theta_k + p,$$ for all $k\leqslant 0$. The last assertion follows from Aubry-Mather theory \cite[Theorem 5.1]{Ban}.
\end{proof}

The following result though not stated this way is present in \cite{Ban2}:
\begin{cor}\label{circlelike}
Let  $(\tilde\theta_i)_{i\in \Z_-}\in \R^{\Z_-}$ be a minimizing sequence and $\rho_0\in \R$ its rotation number given by Proposition \ref{Protnumber}. Let $p$ and $q $ be integers with $q<0$.
\begin{itemize}
\item if $p/q <\rho_0$, then $\tilde\theta_q - \tilde\theta_0 <p$;
\item  if $p/q>\rho_0$, then $\tilde\theta_q - \tilde\theta_0 >p$.
\end{itemize}
In particular $|\tilde\theta_k - \tilde \theta_0 - k\rho_0|<1$ for all $k\leqslant 0$.

\end{cor}

\begin{proof}
Let us prove the first item, the second being similar. Equality $\tilde\theta_q - \tilde\theta_0 =p$  is excluded by the previous result. Let $c\in \R$ and $\tilde u_c$  be a weak K.A.M. solution given by Theorem \ref{caliball} such that $(\tilde\theta_i)_{i\in \Z_-}\in \R^{\Z_-}$ calibrates $\tilde u_c$. And let $p$ and $q $ be integers with $q<0$ such that $p/q <\rho_0$.

Let us now assume that $\tilde\theta_q - \tilde\theta_0 >p>q\rho_0$. As the sequence $(\tilde\theta_{k+q}-p)_{k\leqslant 0}$ also calibrates $\tilde u_c$, it follows from Lemma \ref{orderWK.A.M.} and an induction that $\tilde\theta_{k+q}-p > \tilde\theta_k$ for all $k\leqslant 0$.  Another induction then yields that the sequences $(\tilde\theta_{k+nq}-np)_{n\geqslant 0}$ are increasing. Applying for $k=0$ and dividing by $nq$, we deduce that (recall $q<0$)
$$\frac{\tilde\theta_{nq}}{nq} - \frac{p}{q} <\frac{\tilde\theta_0}{nq}.$$
Letting $n\to +\infty$ the inequality $\rho_0 \leqslant \frac pq$ follows that is a contradiction.

To establish the last assertion, notice that if for some $q\leqslant 0$, $|\tilde\theta_q - \tilde \theta_0 - q\rho_0|\geqslant 1$ then one of the following holds
$$\exists p\in \Z , \quad \tilde\theta_q - \tilde \theta_0 \leqslant p < q\rho_0,$$
$$\exists p\in \Z , \quad \tilde\theta_q - \tilde \theta_0 \geqslant p > q\rho_0.$$
This is not possible by the beginning of the Corollary that was just proved.
\end{proof}

 \begin{nota}
 If $\rho_0\in \R$ and $x\in \R$, we define two numbers: 
 \begin{itemize}
 \item
 $ r^+_{\rho_0} (x)  $ is the largest $r\in \R$ such that $(x,r)\in \mathcal{PG} (c+\tilde u'_c)$ for some weak K.A.M. solution  $\tilde u_c$, associated some $c\in \rho^{-1}(\{\rho_0\})$.
 \item $ r^-_{\rho_0} (x)  $ is the smallest $r\in \R$ such that $(x,r)\in \mathcal{PG} (c+\tilde u'_c)$ for some weak K.A.M. solution  $\tilde u_c$, associated some $c\in \rho^{-1}(\{\rho_0\})$.
 \end{itemize}
 \end{nota}
 
 
 If $\rho_0\in \R$ then  if  $ \rho^{-1}(\{\rho_0\})= [a,b]$, we established that $u_a$ and $u_b$ are unique up to constants. Then  $ r^+_{\rho_0} (x)-b$ is the left derivative of $\tilde u_b$ at $x$ and  $ r^-_{\rho_0} (x)-a$ is the right derivative of $\tilde u_a$ at $x$.

 \begin{nota}
 If $\rho_0\in \R$ and $x\in \R$, we define two numbers $y^+_{\rho_0} (x)$ and $y^-_{\rho_0} (x)$ such that  $y^+_{\rho_0} (x) = \min\{y\in \widetilde {\mathcal{M}} (\rho_0), \ \ y\geqslant x\}$ and $y^-_{\rho_0} (x) = \max\{y\in \widetilde {\mathcal{M}} (\rho_0), \ \ y\leqslant x\}$ where $\widetilde {\mathcal{M}} (\rho_0)$ is  the  lift of the projected Mather set of rotation number $\rho_0$.
Obviously, $x\in [y^-_{\rho_0} (x) , y^+_{\rho_0} (x)]$.

 We denote by $\widetilde{\cm^*\big(\rho_0\big)}$ the lift to $\R^2$ of the Mather set $\cm^*\big(\rho_0\big)$. Then for every 
$c\in  \rho^{-1}(\{\rho_0\})$ and $u_c$   weak K.A.M. solution for $T^c$, $$(y^{\pm }_{\rho_0} (x), p^{\pm }_{\rho_0} (x))=(y^{\pm }_{\rho_0} (x),c+\tilde u'_c\big( y^{\pm }_{\rho_0} (x)))$$ is the unique point of $\widetilde {\mathcal{M}} (\rho_0)$ that is above $y^{\pm }_{\rho_0} (x)$. Moreover, the sequences $\big(y^{\pm n}_{\rho_0} (x)\big)_{n\in \Z}$ where 
$$ y^{\pm n}_{\rho_0} (x) = \pi_1\circ F^n\big( y^{\pm }_{\rho_0} (x), p^{\pm }_{\rho_0} (x)\big)= \pi_1\circ F^n\big( y^{\pm }_{\rho_0} (x), c+ \tilde u'_c\big( y^{\pm n}_{\rho_0} (x)\big)\big)$$
are contained in $\widetilde {\mathcal{M}} (\rho_0)$ and minimizing.

 
 When not necessary, the subscripts $\rho_0$ will be omitted.

 \end{nota}

\begin{propos}\label{coince}
For all $n\in \Z_-$, $y^{-n}(\tilde\theta_0) \leqslant \tilde\theta_n \leqslant  y^{+n}(\tilde\theta_0)$. In particular, $y^\pm ( \tilde\theta_n) = y^{\pm n}(\tilde\theta_0)$.
\end{propos}

\begin{proof}
We use the notations of the proof of the previous Proposition \ref{Protnumber} and recall that $x_k' \leqslant \tilde \theta_k \leqslant x_k$ for all $k\leqslant 0$. Using Lemma \ref{orderWK.A.M.}, a straightforward induction applied to $\tilde u_c$ yields that   for all $n\in \Z_-$, $y^{-n}(\tilde\theta_0) \leqslant x'_n$ and that $x_n \leqslant  y^{+n}(\tilde\theta_0)$. The result follows.

\end{proof}

Next we give a property on minimizing sequences that almost cross twice:

\begin{propos}\label{cross2}
Let $(\tilde\theta_i)_{i\in \Z_-}\in \R^{\Z_-}$ and $(\tilde\theta_i')_{i\in \Z_-}\in \R^{\Z_-}$ be two minimizing sequences. Assume that
\begin{itemize}
\item $\tilde \theta_0 =\tilde \theta_0'$,
\item $\lim\limits _{n\to-\infty} \tilde \theta_n -\tilde \theta_n'=0$,
\item there exists $c\in \R$ and $u_c : \T\to \R$, weak K.A.M. solution for $T^c$ such that $(\tilde\theta_i')_{i\in \Z_-}$ calibrates $\tilde u_c$, meaning that 
$$\forall n<0 , \quad \tilde u_c(\tilde \theta_0') - \tilde u_c(\tilde \theta_n') = \sum_{k=n}^{-1} S(\tilde \theta'_i, \tilde \theta'_{i+1}) +c( \tilde \theta_n' - \tilde \theta_0')-n\alpha(c).$$
\end{itemize}
Then $(\tilde\theta_i,r_i)_{i\in \Z_-}\in \R^{\Z_-}$ calibrates $\tilde u_c$.
\end{propos}

\begin{proof}
Let us argue by contradiction and assume that 
$$\tilde u_c(\tilde \theta_0) - \tilde u_c(\tilde \theta_{n_0}) < \sum_{k=n_0}^{-1} S(\tilde \theta_i, \tilde \theta_{i+1}) +c( \tilde \theta_{n_0} - \tilde \theta_0)-n_0\alpha(c) - 2\varepsilon,$$
 for some $n_0<0$ and $\varepsilon >0$.
 Recall now that if $n<n_0$ then 
 $$\tilde u_c(\tilde \theta_{n_0}) - \tilde u_c(\tilde \theta_{n}) \leqslant  \sum_{k=n}^{n_0-1} S(\tilde \theta_i, \tilde \theta_{i+1}) -(n-n_0)\alpha(c)+c( \tilde \theta_{n} - \tilde \theta_{n_0}) ,$$ 
 and summing we find that 
 $$\forall n<n_0, \quad \tilde u_c(\tilde \theta_0) - \tilde u_c(\tilde \theta_{n}) < \sum_{k=n}^{-1} S(\tilde \theta_i, \tilde \theta_{i+1})-n\alpha(c) +c( \tilde \theta_{n} - \tilde \theta_0) - \varepsilon.$$
 Pick $n<n_0$ such that $|S(\tilde \theta'_{n},\tilde \theta'_{n+1})- S(\tilde \theta_{n},\tilde \theta'_{n+1})|<\varepsilon$ and $| \tilde u_c(\tilde \theta_n')-\tilde u_c(\tilde \theta_n) +c( \tilde \theta_n' - \tilde \theta_n)|<\varepsilon$. We obtain that 
 \begin{multline*}
 S(\tilde \theta_{n},\tilde \theta'_{n+1})+\sum_{k=n+1}^{-1} S(\tilde \theta'_i, \tilde \theta'_{i+1}) <\sum_{k=n}^{-1} S(\tilde \theta'_i, \tilde \theta'_{i+1}) +\varepsilon \\
 =\tilde u_c(\tilde \theta_0') - \tilde u_c(\tilde \theta_n') +n\alpha(c)-c( \tilde \theta_n' - \tilde \theta_0')+\varepsilon 
 \\< \tilde u_c(\tilde \theta_0') - \tilde u_c(\tilde \theta_n)+n\alpha(c) -c( \tilde \theta_n - \tilde \theta_0')+2\varepsilon
 < \sum_{k=n}^{-1} S(\tilde \theta_i, \tilde \theta_{i+1}) .
 \end{multline*}
As $\tilde \theta_0= \tilde \theta'_0$, this contradicts the fact that $(\tilde\theta_i)_{i\in \Z_-}$ is minimizing.

\end{proof}

In order to finish our study, we now discuss in function of the rationality of $\rho_0$.

\subsection{Case where $\rho_0 \notin \Q$} In this case we recover that $\{c_0\} = \rho^{-1} (\{\rho_0\})$ is a singleton and that $u_c$ is unique up to constants. 

For all $x\in \R$, Aubry-Mather theory says that $\big(y^{\pm n}_{\rho_0} (x)\big)_{n\in \Z}$ are orbits of the lift of a circle homeomorphism of rotation number $\rho_0$. Looking back at their definition we see that if they do not coincide,  the circle homeomorphism is a Denjoy counterexample and $\big(y^-(x),y^+(x)\big)$ projects to a wandering interval of this homeomorphism. In all cases, $\lim\limits_{n\to \pm\infty} |y^{-n}(x)-y^{+n}(x)| = 0$.

We deduce the following Theorem that contains Mather's and Bangert's result\footnote{ Recall that for all $i\leqslant 0$ we set $r_i = \frac{\partial S}{\partial \widetilde\Theta} (\tilde\theta_{i-1},\tilde\theta_i)$.} and proves Theorem \ref{solK.A.M.calib} for irrational rotation numbers:

\begin{theorem}\label{min-irr}
 There exists a unique $c_0\in \R$ such that $\rho(c_0) = \rho_0$. Moreover, there exists a unique weak K.A.M. solution $\tilde u_{c_0}$ at cohomology $c_0$ such that $\tilde u_{c_0} (0)=0$. Any minimizing sequence $(\tilde\theta_i)_{i\in \Z_-}\in \R^{\Z_-}$ with rotation number $\rho_0$ calibrates $\tilde u_{c_0}$.
  In particular,
   $(\tilde\theta_i,r_i)_{i<0}\subset \mathcal G (c_0+\tilde u'_{c_0})$ and $(\tilde\theta_0,r_0) \in \pg (c_0+\tilde u'_{c_0})$.
\end{theorem}

\begin{proof}
Let $[a,b] = \rho^{-1}(\{\rho_0\})$. We establish first that if $(\tilde\theta_i)_{i\in \Z_-}\in \R^{\Z_-}$ is a minimizing sequence  with rotation number $\rho_0$, then it calibrates the weak K.A.M. solution $\tilde u_b$.
Let us define for all $n\leqslant 0$, $x_n=\pi_1\circ F^n\big(\tilde \theta_0, r^+(\tilde \theta_0)\big)$\footnote{ See the notations after Corollary \ref{circlelike} for $r^+$.} . It follows that $(x_n)_{n\leqslant 0}$ calibrates $\tilde u_b$ and is therefore minimizing. Arguing as in Proposition \ref{Protnumber} and by Proposition \ref{coince} we discover that
$$\forall n\in \Z_-, \quad y^{-n}(\tilde\theta_0) \leqslant x_n \leqslant  \tilde\theta_n \leqslant  y^{+n}(\tilde\theta_0).$$ 
Using the previous discussion and applying Proposition \ref{cross2} we obtain the result.

Let now $c\in [a,b]$ and $\tilde v : \R \to \R$ be the lift of a  weak K.A.M. solution at cohomology $c$. Let $\tilde \theta_0 \in \R$ be a point of derivability of both $\tilde v$ and $\tilde u_b$. It follows there are respectively a unique sequence $(\tilde\theta_i^c )_{i\in \Z_-}\in \R^{\Z_-}$ calibrating $\tilde v$ and a unique sequence $(\tilde\theta_i^b )_{i\in \Z_-}\in \R^{\Z_-}$ calibrating $\tilde u_b$ such that $\tilde\theta_0=\tilde\theta_0^c=\tilde\theta_0^b$. It follows from the beginning of the proof that $(\tilde\theta_i^c )_{i\in \Z_-}\in \R^{\Z_-}$ calibrates $\tilde u_b$ and by uniqueness that $(\tilde\theta_i^c )_{i\in \Z_-}\in \R^{\Z_-}=(\tilde\theta_i^b )_{i\in \Z_-}\in \R^{\Z_-}$. As a consequence, the equality
$$b+\tilde u_b'(\tilde\theta_0) = c+ \tilde v'(\tilde\theta_0) =  \frac{\partial S}{\partial \Theta} (\tilde\theta^c_{-1},\tilde\theta_0) $$ 
is obtained. As this equality holds almost everywhere, we conclude that
$$b = \int_0^1 \big(b+\tilde u_b'(\tilde\theta)\big) d \tilde\theta =\int_0^1\big( c+ \tilde v'(\tilde\theta)\big) d\tilde\theta = c,$$ 
and then that $\tilde u_b - \tilde v$ is constant.
\end{proof}

\subsection{Case where $\rho_0\in \Q$} In this case we denote $[a,b]=\rho^{-1} (\{\rho_0\})$ and we write $\rho_0 = p/q$ in irreducible form with $q> 0$. It follows that for all $x\in \R$, $\big(y^{\pm n}_{\rho_0} (x)\big)_{n\in \Z} =\big(y^{\pm n+q}_{\rho_0} (x)+p\big)_{n\in \Z} $. Moreover, we have seen previously that $u_a$ and $u_b$ are unique up to constants. Hereafter, unless specified otherwise, $(\tilde \theta_i)_{i\leqslant 0} $ is a minimizing sequence with rotation number $p/q$.

In the spirit of Aubry-Mather theory, we start by a property of non-crossing of a minimizing sequence with its translates, in the spirit of Corollary \ref{circlelike}:

\begin{propos}\label{noncrossing}
Assume $(\tilde \theta_i)_{i\leqslant 0} $ does not verify  $(\tilde \theta_i)_{i\leqslant 0}  = (\tilde \theta_{i-q}-p)_{i\leqslant 0}$. One of the two holds:
\begin{itemize}
\item $ \tilde \theta_0 \leqslant\tilde \theta_{-q}-p$, $ \tilde \theta_i < \tilde\theta_{i-q}-p$ for all $i<0$ and $\lim\limits_{i\to-\infty} |\tilde\theta_i -y^{+i}(\tilde\theta_0)  | = 0$.
\item $ \tilde \theta_0 \geqslant \tilde\theta_{-q}-p$, $ \tilde \theta_i > \tilde\theta_{i-q}-p$ for all $i<0$ and $\lim\limits_{i\to-\infty} |\tilde\theta_i -y^{-i}(\tilde\theta_0)  | = 0$.
\end{itemize}

\end{propos}

\begin{proof}
If $(\tilde \theta_i)_{i\leqslant 0} $ and $ (\tilde \theta_{i-q}-p)_{i\leqslant 0}$ do not cross, we set $i_0=0$. Otherwise, there exists $i_0\leqslant 0$ such that  $(\tilde \theta_i)_{i\leqslant 0} $ and $ (\tilde \theta_{i-q}-p)_{i\leqslant 0}$ cross either at $i_0$ or between $i_0 $ and $i_0+1$ for some $i_0<0$. As two minimizing sequences cross at most once, it follows that in all the previous cases, either $ \tilde \theta_i < \tilde\theta_{i-q}-p$ for all $i<i_0$ or $ \tilde \theta_i > \tilde\theta_{i-q}-p$ for all $i<i_0$. Let us assume the first holds, the second case is treated similarly. Then for $i<i_0$ the sequence $(\tilde\theta_{i-kq}-kp)_{k\geqslant 0}$ is increasing. Moreover, by Proposition \ref{coince}, $\tilde\theta_{i-kq}-kp \leqslant y^{+i-kq}(\tilde\theta_0)-kp =  y^{+i}(\tilde\theta_0)$. Hence the limit $z_i = \lim\limits_{k\to+\infty} \tilde\theta_{i-kq}-kp$ exists and verifies that $z_{i-q}-p = z_i$. As this sequence is minimizing (as a limit of minimizing sequences) and projects to a periodic sequence on the circle, we deduce that $(z_i)_{i<i_0}\subset  \widetilde {\mathcal{M}} (\rho_0)$ and then $(z_i)_{i<i_0} = y^{+i}(\tilde\theta_0)$ by definition of $y^{+i}(\tilde\theta_0)$. 

If now $(\tilde \theta_i)_{i\leqslant 0} $ and $ (\tilde \theta_{i-q}-p)_{i\leqslant 0}$ cross either at $i_0$ or between $i_0 $ and $i_0+1$ for some $i_0<0$, as we have also proven they are $\alpha$-asymptotic, we obtain a contradiction with \cite[Lemma 3.9]{Ban}.

\end{proof}

The following result ends the proof of Theorem \ref{caliball} and also gives precisions to Proposition \ref{noncrossing} by excluding the possibility of an equality  $ \tilde \theta_0 = \tilde \theta_{-q}-p$ under its hypotheses.

\begin{propos}\label{minrat}
Using the previous notations, if $\lim\limits_{i\to-\infty} |\tilde\theta_i -y^{+i}(\tilde\theta_0)  | = 0$ then $(\tilde \theta_i)_{i\leqslant 0} $ calibrates  $\tilde u_b$. 

 If $\lim\limits_{i\to-\infty} |\tilde\theta_i -y^{-i}(\tilde\theta_0)  | = 0$ then $(\tilde \theta_i)_{i\leqslant 0} $ calibrates  $\tilde u_a$. 
 
 In particular, $\tilde \theta_0 = \tilde\theta_{-q}-p$ if and only if $(\tilde \theta_i)_{i\leqslant 0}  = (\tilde \theta_{i-q}-p)_{i\leqslant 0}$.
 
 In all cases, 
$(\tilde\theta_i,r_i)_{i<0}\subset \mathcal G (a+\tilde u'_a)\cup \mathcal G (b+\tilde u'_b) $ and $(\tilde\theta_i,r_i)_{i\leqslant 0}\subset \mathcal{P G} (a+\tilde u'_a)\cup \mathcal{P G} (b+\tilde u'_b) $.

\end{propos}

\begin{proof}
Let us prove the first assertion, the rest is done in a similar way. If  $\lim\limits_{i\to-\infty} |\tilde\theta_i -y^{+i}(\tilde\theta_0)  | = 0$, let us define for all $n\leqslant 0$, $x_n=\pi_1\circ F^n\big(\tilde \theta_0, r^-(\tilde \theta_0)\big)$. It follows that $(x_n)_{n\leqslant 0}$ calibrates $\tilde u_a$ and is therefore minimizing. Arguing as in Proposition \ref{Protnumber} and by Proposition \ref{coince} we discover that for all $n\in \Z_-$, $y^{-n}(\tilde\theta_0) \leqslant   \tilde\theta_n \leqslant x_n \leqslant  y^{+n}(\tilde\theta_0)$. 
Applying Proposition \ref{cross2} we obtain the result.

The next assertion in the Theorem is a direct consequence of Corollary \ref{permin}. The rest follows from general properties of weak K.A.M. solutions.
 
\end{proof}

 In particular we have finished proving Theorem \ref{solK.A.M.calib} establishing more precisely that the cohomology can be taken to be $a$ or $b$.\\

Next, the reciprocal question of existence of minimizing sequences verifying certain conditions is addressed:

\begin{propos}\label{existence+-}
For all $\tilde\theta_0\in \R$ there exist two minimizing sequences $(\tilde \theta_i^\pm)_{i\leqslant 0} $
with rotation number $p/q$ such that $\tilde\theta_0^+ = \tilde\theta_0^- = \tilde \theta_0$ and 
\begin{itemize}
\item $\lim\limits_{i\to-\infty} |\tilde\theta_i^+ -y^{+i}(\tilde\theta_0)  | = 0$ (and $(\tilde \theta_i^+)_{i\leqslant 0}$ calibrates $\tilde u_a$);
\item $\lim\limits_{i\to-\infty} |\tilde\theta_i^- -y^{-i}(\tilde\theta_0)  | = 0$ (and $(\tilde \theta_i^-)_{i\leqslant 0}$ calibrates $\tilde u_b$);
\end{itemize}

\end{propos}

\begin{proof}
If $\tilde \theta_0\in  \widetilde {\mathcal{M}} (p/q)$ then by Aubry-Mather theory, $y^{+i}(\tilde\theta_0) = y^{-i}(\tilde\theta_0)$ for all $i\leqslant 0$ and $\big(y^{\pm i}(\tilde\theta_0) \big)_{i\leqslant 0}$ is the only minimizing orbit starting at $\tilde \theta_0$ with rotation number $p/q$.

We handle now the other and more interesting case where $y^{-0}(\tilde\theta_0) < \tilde \theta_0< y^{+0}(\tilde\theta_0) $ and prove the existence of $(\tilde \theta_i^-)_{i\leqslant 0}$ as the existence of $(\tilde \theta_i^+)_{i\leqslant 0}$ is established in a very similar way.

Let $(c_n)_{n> 0}$ be a decreasing sequence of real numbers converging to $b$. Setting for $n>0$, $\rho_n = \rho(c_n)$ it follows that $(\rho_n)_{n>0}$ is nonincreasing, converges to $p/q$ and that $\rho_n>p/q$ for all $n>0$. For all $n>0$, let $(\tilde\theta_k^n)_{k\leqslant 0}$ be a minimizing sequence with rotation number $\rho_n$ and such that $\tilde\theta_0^n = \tilde\theta_0$ (take any calibrating sequence for a weak K.A.M. solution at cohomology $c_n$ starting at $ \tilde\theta_0$). Up to extracting, we may assume  that for all $k\leqslant 0$, the sequence $(\tilde\theta_k^n)_{n>0}$  converges to a $\tilde\theta_k^-$. It follows that $(\tilde\theta_k^-)_{k\leqslant 0}$ is minimizing, has rotation number $p/q$ and verifies $\tilde\theta_0^- = \tilde\theta_0$. 

Applying Corollary \ref{circlelike} yields the inequalities
$$\forall n>0, \quad \tilde\theta^n_{-q} >\tilde\theta_0+p.$$ 
Passing to the limit we get $\tilde\theta_{-q}^->\tilde\theta_0+p$, as equality is prohibited by Corollary \ref{permin}. The rest now follows from Proposition \ref{noncrossing}.
\end{proof}

 This leads to a reciprocal to Proposition \ref{minrat}:
 
 \begin{theorem}\label{equivration}
 Let $\tilde\theta_0 \notin \widetilde{\mathcal M}(p/q)$ and $(\tilde\theta_i)_{i\leqslant 0}$ a minimizing sequence with rotation number $p/q$. The following assertions are equivalent:
 \begin{enumerate}
 \item $ \tilde \theta_0 <\tilde \theta_{-q}-p$ (resp. $ \tilde \theta_0 >\tilde \theta_{-q}-p$);
 \item $\lim\limits_{i\to-\infty} |\tilde\theta_i -y^{+i}(\tilde\theta_0)  | = 0$ (resp. $\lim\limits_{i\to-\infty} |\tilde\theta_i -y^{-i}(\tilde\theta_0)  | = 0$);
 \item $(\tilde\theta_i)_{i\leqslant 0}$ calibrates $\tilde u_a$ (resp. $(\tilde\theta_i)_{i\leqslant 0}$ calibrates $\tilde u_b$).
 \end{enumerate}

 \end{theorem}

\begin{proof}
The only thing left to prove is that (3) implies (1). 

Let us assume in a first step that $\tilde\theta_0$ is a point of derivability of  $\tilde u_a$. Let $(\tilde \theta_i^\pm)_{i\leqslant 0} $ be the sequences given by Proposition \ref{existence+-} and let $(r_i^\pm)_{i\leqslant 0}$ be the associated sequences such that $(\tilde \theta_i^\pm,r_i^\pm)_{i\leqslant 0} $ are orbits of $F$. As $(\tilde \theta_i^+)_{i\leqslant 0} $ calibrates $\tilde u_a$, we discover that $r_0^+ = a+\tilde u_a'(\tilde\theta_0)$.

Assume now $\tilde\theta_0$ is arbitrary and $\{\tilde\theta_0\} \times [R_0,R_0'] = \partial \tilde u_a(\tilde\theta_0)$. Let $(\tilde\theta^n_0)_{n\in \N}$ be a decreasing sequence converging to $\tilde\theta_0$, of derivability points of $\tilde u_a$. It follows that $\tilde u'_a(\tilde\theta_0^n ) \to R_0'$.
For all $n\geqslant 0$, let $(\tilde\theta^n_i)_{i\leqslant 0}$ be the unique sequence calibrating $\tilde u_a$, starting at $\tilde \theta_0^n$. Finally, let $(\tilde\theta'_i)_{i\leqslant 0}$ be the limit of the sequences  $(\tilde\theta^n_i)_{i\leqslant 0}$ so that for all $i\leqslant 0$, $\tilde\theta_i = \pi_1\circ F^i (\tilde\theta_0,a+R'_0)$. Thanks to the beginning of the proof, for all $n\geqslant 0$, the inequality $ \tilde \theta_0^n \leqslant \tilde \theta_{-q}^n-p$. Passing to the limit we discover that 
$ \tilde \theta_0 <\tilde \theta'_{-q}-p$, the inequality being strict as $\tilde\theta_0 \notin \widetilde{\mathcal M}(p/q)$.

If now $(\tilde\theta_i)_{i\leqslant 0}$ is any minimizing sequence starting at $\tilde\theta_0$ that calibrates $\tilde u_a$ and $(\tilde\theta_i,r_i)_{i\leqslant 0}$ the associated orbit of $F$ we know that $(\tilde\theta_i)_{i\leqslant 0}$ and $(\tilde\theta'_i)_{i\leqslant 0}$ can only cross at $0$. As $r_0 \in [R_0, R_0']$, the twist condition implies that $\tilde\theta_{-1} \geqslant \tilde\theta'_{-1}$. We therefore conclude that $\tilde\theta_0 <\tilde\theta'_{-q}-p \leqslant \tilde\theta_{-q}-p $ that was to be proven.

\end{proof}

We deduce a further property concerning the pseudographs of $u_a$ and $u_b$:

\begin{propos}\label{ordrerat}
The full pseudographs 
$\pg(a+u_a)$ and $\pg(b+u_b)$ only intersect above $\cm(p/q)$.
\end{propos}

\begin{proof}
Let $\tilde\theta_0 \notin \widetilde{\mathcal M}(p/q)$. We denote  $\{\tilde\theta_0\} \times [r_0^a,R_0^{a}] = \partial \tilde u_a(\tilde\theta_0)$ and  $\{\tilde\theta_0\} \times [r_0^b,R_0^{b}] = \partial \tilde u_b(\tilde\theta_0)$. We will prove that $b+r_0^b> a+R_0^a$ thus establishing that  $\pg(b+\tilde u_b)$ is strictly above $\pg(a+\tilde u_a)$ on the interval $\big(y^-(\tilde\theta_0), y^+(\tilde\theta_0)\big)$. We set $\tilde\theta_0^a=\tilde\theta_0^b=\tilde\theta_0$ and for $i<0$,
$\tilde\theta_i^a = \pi_1\circ F^i(\tilde\theta_0, a+R_0^a)$ and $\tilde\theta_i^b = \pi_1\circ F^i(\tilde\theta_0, b+r_0^b)$. From the previous Theorem \ref{equivration} we know that $\lim\limits_{i\to-\infty} |\tilde\theta_i^a -y^{+i}(\tilde\theta_0)  | = 0$ and that $\lim\limits_{i\to-\infty} |\tilde\theta_i^b -y^{-i}(\tilde\theta_0)  | = 0$. Therefore for $i$ small enough, $\tilde\theta_i^a>\tilde\theta_i^b$. As both minimizing sequences only can cross at $0$, it follows that $\tilde\theta_{-1}^a>\tilde\theta_{-1}^b$. We conclude from the twist condition that $b+r_0^b> a+R_0^a$.
\end{proof}

As a Corollary we recover a famous result of Mather and Bangert:

\begin{cor}
The set $\rho^{-1}(p/q)$ is a singleton if and only if $\widetilde{\mathcal{M}}(p/q) = \T$.
\end{cor}

\begin{proof}
If $\widetilde{\mathcal{M}}(p/q) = \T$ then there is an invariant, $1$--periodic Lipschitz graph $\tilde\theta \mapsto r_{\tilde\theta}$ for all $c\in \rho^{-1}(p/q)$, if $\tilde u : \R \to \R$ is a corresponding weak K.A.M. solution, it is of class $C^1$ and $c+\tilde u'(\tilde\theta) = r_{\tilde\theta}$ for all $\tilde\theta\in \R$. It follows that $c = \int_0^1 r_x d x$ is unique.

Reciprocally, if $\widetilde{\mathcal{M}}(p/q) \neq \T$, thanks to the preceding Proposition \ref{ordrerat}, there is $\tilde\theta_0 \notin  \widetilde{\mathcal{M}}(p/q)$ such that $\tilde u_a$ and $\tilde u_b$ both are derivable at $\tilde\theta_0$ with $a+\tilde u'_a(\tilde\theta_0) < b + \tilde u'_b(\tilde\theta_0)$ and by property of semi--concave functions, this inequality is strict in a neighborhood of $\tilde\theta_0$. As $a+\tilde u'_a(\tilde\theta) \leqslant  b + \tilde u'_b(\tilde\theta)$ holds almost everywhere, integrating on $[0,1]$ we find $a<b$.
\end{proof}

As a conclusion we obtain the following property on weak K.A.M. solutions:

\begin{theorem}\label{K.A.M.f}
Let $c\in [a,b]$ and $u_c$ be a weak K.A.M. solution for $T^c$. Let $\big(\tilde \theta_0, c+\tilde u'_c(\tilde \theta_0)\big) \in \mathcal G (c+\tilde u'_c)$ and $(\tilde \theta_i)_{i\leqslant 0} $ the associated minimizing sequence that calibrates $\tilde u_c$. 
\begin{itemize}
\item If $\lim\limits_{i\to -\infty} \tilde \theta_i - y^{-i}(\tilde \theta_0)=0$ then $\big(\tilde \theta_0, c+\tilde u'_c(\tilde \theta_0)\big) \in \mathcal G (b+\tilde u'_b)$.
\item If $\lim\limits_{i\to -\infty} \tilde \theta_i - y^{+i}(\tilde \theta_0)=0$ then $\big(\tilde \theta_0, c+\tilde u'_c(\tilde \theta_0)\big) \in \mathcal G (a+\tilde u'_a)$.
\end{itemize}
 
\end{theorem}

\begin{proof}
Let us  recall that for a semi-concave function $v : \R \to \R$, if $v'(x_0)$ exists then $x_0$ is a continuity point of $x\mapsto \partial v(x)$. On the other side, if $N\subset \R$ has Lebesgue measure $0$ and contains the nonderivable points of $v$,  and if $\lim\limits_{\substack{x\to x_0 \\ x\notin N}} v'(x) = p $ exists, then $v'(x_0) = p$ exists.

Coming back to the Theorem, the result obviously holds is $\tilde\theta_0 \in  \widetilde {\mathcal{M}} (\rho_0)$. We now assume otherwise.

Let us introduce $N\subset \R$ to be the countable set of points where either $\tilde u_a$, $\tilde u_b$ or $\tilde u_c$ is not derivable. 

Let us prove the first point. In this case, by Proposition \ref{noncrossing}, $\tilde \theta_{-q} -p < \tilde\theta_0$. There exists $\varepsilon >0$ such that if $\tilde\theta'_0 \in \R\setminus N$ verifies $|\tilde\theta_0 -\tilde \theta'_0|<\varepsilon$, then $\tilde \theta'_{-q} -p < \tilde\theta'_0$ where 
$\tilde \theta'_i = \pi_1\circ F^i\big( \tilde \theta_0', c+\tilde u'_c(\tilde \theta'_0) \big)$.  Up to taking $\varepsilon$ smaller, then $\tilde \theta_0'\in \big(y^{-}(\tilde \theta_0),y^{+}(\tilde \theta_0)\big)$, it then follows from Proposition \ref{noncrossing} and Theorem \ref{minrat} that $(\tilde \theta'_i)_{i\leqslant 0}$ calibrates $\tilde u_b$ and then that 
$$\forall i\leqslant 0,\quad \tilde\theta_i'= \pi_1\circ F^i\big( \tilde \theta_0', b+\tilde u'_b(\tilde \theta'_0) \big)    . $$
Finally, we have established that 
$$\lim\limits_{\substack{\tilde \theta'_0\to \tilde\theta_0 \\ \tilde \theta'_0\notin N}} b+\tilde u'_b(\tilde \theta'_0)  =\lim\limits_{\substack{\tilde \theta'_0\to \tilde\theta_0 \\ \tilde \theta'_0\notin N}} c+\tilde u'_c(\tilde \theta'_0) = c+\tilde u'_c(\tilde \theta_0)  ,$$
that proves our result.
\end{proof}

 The following  corollary ends the proof of Theorem \ref{TPlacesolK.A.M.}.

\begin{cor}\label{inclusioncle}
Let $c\in [a,b]$ and $u_c$ be a weak K.A.M. solution for $T^c$ then 
$$\mathcal G (c+\tilde u'_c) \subset \mathcal G (a+\tilde u'_a)\cup \mathcal G (b+\tilde u'_b),$$
$$\overline{\mathcal G (c+\tilde u'_c)} \subset \overline{ \mathcal G (a+\tilde u'_a)} \cup \overline{\mathcal G (b+\tilde u'_b)}.$$
\end{cor}

We may also provide a description of what  $\mathcal {PG} (c+\tilde u'_c)$ looks like, similar to the classical example of the pendulum. The open region between $\mathcal {PG} (a+\tilde u'_a)$ and $\mathcal {PG} (b+\tilde u'_b)$ has a connected component between two consecutive points of the projected Mather set that  projects on an interval $(y^-,y^+)$. Then, either  $\mathcal {PG} (c+\tilde u'_c)$ coincides with  $\mathcal {PG} (a+\tilde u'_a)$ on  $(y^-,y^+)$, either $\mathcal {PG} (c+\tilde u'_c)$ coincides with  $\mathcal {PG} (b+\tilde u'_b)$ on  $(y^-,y^+)$, either there exists $z\in (y^-,y^+)$ such that $\mathcal {PG} (c+\tilde u'_c)$ coincides with  $\mathcal {PG} (b+\tilde u'_b)$ on  $(y^-,z)$ and $\mathcal {PG} (c+\tilde u'_c)$ coincides with  $\mathcal {PG} (a+\tilde u'_a)$ on  $(z,y^+)$.

\subsection{Pseudographs of weak K.A.M. solutions}

We end by a crucial property of pseudo-graphs associated to weak K.A.M. solutions that we believe is of independent interest:

\begin{propos}\label{imageWK.A.M.}
Let $c\in \R$ and $u_c : \T\to \R$ be a weak K.A.M. solution for $T^c$. Then $f^{-1}\big( \mathcal{PG} (c+ u'_c)\big)$ is the graph of a continuous function.
\end{propos}

\begin{proof}
Recall that by \cite{Arna2}, $ \mathcal{PG} (c+ u'_c)$ is a Lipschitz embedded circle (see Lemma \ref{Lfullman}) that we can parametrize by a map $\gamma : \T \to \A$. Is $\tilde \gamma : \R \to \R^2$ is a lift of $\gamma$ we assume without loss of generality that  $\pi_1\circ \tilde \gamma$ is nondecreasing.

If $t<t'$ are real numbers  we consider two cases: if $\pi_1\circ \tilde \gamma (t) = \pi_1\circ \tilde \gamma (t')$ then  $\pi_2\circ \tilde \gamma (t) > \pi_2\circ \tilde \gamma (t')$ (because $\tilde u_c$ is semi-concave) and   $\pi_1\circ F^{-1}\big(\tilde \gamma (t)\big)< \pi_1\circ F^{-1}\big( \tilde \gamma (t')\big)$ because of the twist condition. 

If now $\pi_1\circ \tilde \gamma (t) < \pi_1\circ \tilde \gamma (t')$ then we consider $t\leqslant t_1 < t_2 \leqslant t'$ such that $\pi_1\circ \tilde \gamma (t) = \pi_1\circ \tilde \gamma (t_1)$, $\pi_1\circ \tilde \gamma (t') = \pi_1\circ \tilde \gamma (t_2)$ and $\{ \tilde \gamma (t_1), \tilde \gamma (t_2) \} \in  \overline{\cg (c+\tilde u_c')}$. It follows that $\pi_2\circ \tilde \gamma (t) > \pi_2\circ \tilde \gamma (t_1)$ and $\pi_2\circ \tilde \gamma (t_2) > \pi_2\circ \tilde \gamma (t')$. Moreover, we deduce from Lemma \ref{orderWK.A.M.} that $\pi_1\circ F^{-1}\big(\tilde \gamma (t_1)\big) \leqslant \pi_1\circ F^{-1}\big( \tilde \gamma (t_2)\big)$. Moreover, as $F^{-1}\big(\overline{\cg (c+\tilde u_c')}\big)\subset \cg (c+\tilde u_c')$ the previous inequality is strict. We conclude that

$$\pi_1\circ F^{-1}\big(\tilde \gamma (t)\big)\leqslant     \pi_1\circ F^{-1}\big(\tilde \gamma (t_1)\big) < \pi_1\circ F^{-1}\big( \tilde \gamma (t_2)\big)  \leqslant  \pi_1\circ F^{-1}\big( \tilde \gamma (t')\big).$$
We have established that the function $\pi_1\circ F^{-1}\circ \tilde\gamma$ is increasing and that proves the Proposition.
\end{proof}

 \begin{remk}
The preceding result may be interpreted in terms of positive Lax-Oleinik maps. Indeed if one defines $T_+^c \tilde u_c (x) = \max\limits_{x'\in \R}  \tilde u_c (x') - S(x,x') +c(x'-x)$ then one may deduce that $T_+^c \tilde u_c$ is a $C^1$ function and that $F^{-1}\big( \mathcal{PG} (c+\tilde u'_c)\big) = \mathcal{G} (c+T_+^c \tilde u'_c)$. This is related to Lasry-Lyons type results, see \cite{Be4,Be5,Za,Za2}.

 \end{remk}
 
In Aubry-Mather Theory, it is known that given a rotation number $\rho$, on each vertical $V_\theta$ there is
\begin{itemize}
\item at most one bi-infinite minimizing orbit of rotation number $\rho$ intersecting $V_\theta$ if $\rho\notin \Q$,
\item at most two bi-infinite minimizing orbit of rotation number $\rho$ intersecting $V_\theta$ if $\rho\in  \Q$,
\end{itemize}
in the latter case if there are two, one is $\alpha$-asymptotic to $\big(y^{-i}(\theta)\big)_{i\in \Z}$ and $\omega$-asymptotic to  $\big(y^{+i}(\theta)\big)_{i\in \Z}$ and the other is $\omega$-asymptotic to $\big(y^{-i}(\theta)\big)_{i\in \Z}$ and $\alpha$-asymptotic to  $\big(y^{+i}(\theta)\big)_{i\in \Z}$.

In our study of one-sided infinite minimizing sequence we obtain as a corollary a similar statement. The only difference is that instead of taking as reference the vertical foliation, we take its image by $f$. 

As an application of the previous Theorem we obtain:

\begin{theorem}\label{bang}
Let $\theta\in \T$ and $\rho_0 \in \R$, then
\begin{itemize}
\item if $\rho_0\notin \Q$, there exists at most one $(x,p)\in f(V_\theta)$ such that $\big(\pi_1\circ f^i(x,p)\big)_{i\in \Z_-}$ is minimizing with rotation number $\rho_0$;
\item if $\rho_0\in \Q$, there exists at most two $(x,p)\in f(V_\theta)$ such that $\big(\pi_1\circ f^i(x,p)\big)_{i\in \Z_-}$ is minimizing with rotation number $\rho_0$.
\end{itemize}
In the latter case, if there are two such points $(x_1,p_1) $ and $(x_2,p_2)$ with $x_1<x_2$ then 
$\big(\pi_1\circ f^i(x_1,p_1)\big)_{i\in \Z_-} $ is $\alpha$-asymptotic to $\big(y^{+(i-1)}(\theta)\big)_{i\in \Z}$ and $\big(\pi_1\circ f^i(x_2,p_2)\big)_{i\in \Z_-} $ is $\alpha$-asymptotic to $\big(y^{-(i-1)}(\theta)\big)_{i\in \Z}$.
\end{theorem}

\begin{proof}
If $\rho_0 \notin \Q$ the only possible such point is $f\big(V_\theta \cap  f^{-1}\big( \mathcal{PG} (c+ u'_c)\big)\big)$ where $\{c\} = \rho^{-1}(\{\rho_0\})$.

If $\rho \in \Q$ the only possible such points are $f\big(V_\theta \cap  f^{-1}\big( \mathcal{PG} (a+ u'_a)\big)\big)$ and $f\big(V_\theta \cap  f^{-1}\big( \mathcal{PG} (b+ u'_b)\big)\big)$ where $[a,b] = \rho^{-1}(\{\rho_0\})$.

\end{proof}
 This proves the end of Theorem \ref{Ttwistvertical}.

 \appendix
 
 \section{Examples}\label{AppA}
 \subsection{An example a semi-concave function that is not a weak K.A.M. solution for $\widehat T^c$ and that satisfies $f^{-1}\big(\overline{\cg(c+u')}\big)\subset \cg(c+u')$}\label{sspseudo}
 Let us begin by introducing $g_t:\A\rightarrow\A$ as being the time $t$ map  of the Hamiltonian flow of the double pendulum Hamiltonian
 $$H(\theta, r)=\frac{1}{2}r^2+\cos (4\pi\theta).$$
 If  $t>0$ is small enough, $g_t$ is an ESTwD.\\
 Observe that $H$ is a so-called Tonelli Hamiltonian (see \cite{Fa3} for the definition) with associated Lagrangian $L(\theta, v)=\frac{1}{2}v^2-\cos (4\pi\theta)$. The global minimum $-1$ of $L$ is attained in $(0, 0)$ and $(\frac{1}{2}, 0)$.\\
 If $G_t$ is the time $t$ map of the lift of $H$ to $\R^2$, then $G_t$ is a lift of $g_t$ and if $G_s(\theta, r)=(\theta_s, r_s)$, a generating function of $G_t$ is 
 $$S_t(\theta, \theta_t)=\int_0^tL(\theta_s, \dot\theta_s)ds.$$
 By using this formula, observe that the only ergodic minimizing measures for the cohomology class $0$ are the Dirac measure at $0$ and $\frac{1}{2}$.\\

 Then we denote by $h:\A\rightarrow \A$ the map that is defined by $h(\theta, r)=(\theta+\frac{1}{2}, r)$. Then $f=h\circ g_t=g_t\circ h$  is again an ESTwD and $H$ is an  integral for $f$, which means that $H\circ f=H$.\\
 It is easy to check that a generating function of a lift $F$ of $f$ is given by
 $$S(\theta, \Theta)= S_t\big(\theta, \Theta-\frac{1}{2}\big).$$
 From this, we deduce that the Mather set corresponding to the cohomology class zero (and the rotation number $\frac{1}{2}$) is the support of a unique ergodic measure, that is the mean of two Dirac measure $\frac{1}{2}(\delta_{(0, 0)}+\delta_{(\frac{1}{2}, 0)})$.\\
 As there is only one such minimizing measure, we know that there is a unique, up to constants, weak K.A.M. solution $u$ with cohomology class $0$. But there are a lot of graphs of $v'$ with $v:\T\rightarrow \R$ semi--concave that are invariant by $f$. The first one we draw corresponds to the weak K.A.M. solution whose graph is strictly mapped into itself by   $f^{-1}$. Perturbing slightly the pseudograph in the level $\{ H=1\}$, we obtain another backward invariant pseudograph that doesn't correspond to a weak K.A.M. solution.
 
 In the right drawing \ref{pgnwk}, the perturbation of the pseudograph must be small enough so that, in the right eye on the upper manifold, the piece of pseudograph that goes beyond the vertical dotted line is mapped  by $f^{-1}$ in the upper piece of pseudograph of the left eye. With the notations of the figure, $f^{-1}\big(d, s^+(d)\big) = \big(e,s^+(e)\big)$.
 

 \begin{figure}[h!]
 \begin{minipage}[c]{.46\linewidth}

  \begin{tikzpicture}[scale = .6]
    \draw[domain=0:4][samples=250] plot (2
   *\x,2* abs{sin(pi*\x/2 r)} );
    \draw[domain=0:4][samples=250] plot (2
   *\x,-2* abs{sin(pi*\x/2 r)} );
   
   \draw  (0,-3)--(0,3) ;
     \draw  (8,-3)--(8,3) ;
     \draw (0,0) -- (8,0);

\draw (0,0) node[left]{0} ;
\draw (8,0) node[right]{1} ;

  \draw[color=red][line width=2pt][domain=0:1][samples=250] plot (2
   *\x,2* abs{sin(pi*\x/2 r)} );
    \draw[color=red][line width=2pt][domain=1:2][samples=250] plot (2
   *\x,-2* abs{sin(pi*\x/2 r)} );

  \draw[color=red][line width=2pt][domain=2:3][samples=250] plot (2
   *\x,2* abs{sin(pi*\x/2 r)} );
    \draw[color=red][line width=2pt][domain=3:4][samples=250] plot (2
   *\x,-2* abs{sin(pi*\x/2 r)} );

\draw [color=red][dashed] [line width=2pt](2,2* abs{sin(pi*1/2 r)})--(2,-2* abs{sin(pi*1/2 r)}) ;
\draw [color=red] [dashed][line width=2pt](6,2* abs{sin(pi*1/2 r)})--(6,-2* abs{sin(pi*1/2 r)}) ;

\end{tikzpicture}

\caption{The  pseudograph of the weak K.A.M. solution at cohomology $0$}

\end{minipage} \hfill
\begin{minipage}[c]{.46\linewidth}

  \begin{tikzpicture}[scale = .6]
    \draw[domain=0:4][samples=250] plot (2
   *\x,2* abs{sin(pi*\x/2 r)} );
    \draw[domain=0:4][samples=250] plot (2
   *\x,-2* abs{sin(pi*\x/2 r)} );
   
   \draw  (0,-3)--(0,3) ;
     \draw  (8,-3)--(8,3) ;
     \draw (0,0) -- (8,0);

\draw (0,0) node[left]{0} ;
\draw (8,0) node[right]{1} ;

  \draw[color=red][line width=2pt][domain=0:.8][samples=250] plot (2
   *\x,2* abs{sin(pi*\x/2 r)} );
    \draw[color=red][line width=2pt][domain=.8:2][samples=250] plot (2
   *\x,-2* abs{sin(pi*\x/2 r)} );

  \draw[color=red][line width=2pt][domain=2:3.2][samples=250] plot (2
   *\x,2* abs{sin(pi*\x/2 r)} );
    \draw[color=red][line width=2pt][domain=3.2:4][samples=250] plot (2
   *\x,-2* abs{sin(pi*\x/2 r)} );

\draw [color=red][dashed] [line width=2pt](1.6,2* abs{sin(pi*.4 r)})--(1.6,-2* abs{sin(pi*.4 r)}) ;
\draw [color=red] [dashed][line width=2pt](6.4,2* abs{sin(pi*.4 r)})--(6.4,-2* abs{sin(pi*.4 r)}) ;

\draw [dashed] [line width=1pt](2,2* abs{sin(pi*1/2 r)})--(2,-2* abs{sin(pi*1/2 r)}) ;
\draw  [dashed][line width=1pt](6,2* abs{sin(pi*1/2 r)})--(6,-2* abs{sin(pi*1/2 r)}) ;

\draw (6.4,0) node[below right] {$d$} node{$\bullet$};
\draw (1.2,0) node[below left] {$e$} node{$\bullet$};

\draw (6.4,2*abs{sin(pi*.4 r)}) node[above] {{\small {$(d,s^+(d))$}}} node{$\bullet$};
\draw (1.2,2*abs{sin(pi*.3 r)})  node{$\bullet$};
\draw (1.2,.25+2*abs{sin(pi*.3 r)}) node[above] {{\small$(e,s^+(e))$}} ;

\end{tikzpicture}

\caption{A backward invariant pseudograph   that is not a weak K.A.M. solution}
\label{pgnwk}
\end{minipage}
\end{figure}

   \begin{remk}
   The previous example also shows that Corollary \ref{inclusioncle} is not an equivalence in the sense that if the pseudograph of a semi--concave  function $c+u'_c$ satisfies the inclusions of Corollary \ref{inclusioncle}, then $u_c$ is not necessarily a weak K.A.M. solution at cohomology $c$.
   \end{remk}

 \subsection{Cases where the discounted solution doesn't depend continuously on $c$}\label{sscomplint}
 
 Let us start this appendix of counterexamples with a positive result. We will show that even if discounted solutions may depend in a discontinuous way on $c$, the same is not true for their derivative. In what follows we use the notion of Clarke sub-derivative introduced earlier
  in Definition \ref{Clarke}.

Let us recall that by Proposition \ref{hausdorff},
if $g_n : \T\to \R$ are equi-semi-concave functions  converging to $g : \T \to \R$, then $\mathcal{PG} (g_n')$ converges to $\mathcal{PG} ( g')$ for the Hausdorff distance.

Let us now state our result:
\begin{propos}\label{discount}
Let $f:\A \to \A$ be an ESTwD. For $c\in \R$, we denote by $\cu_c$ the weak K.A.M. discounted solution. Then the map $c\mapsto \mathcal{PG} ( \cu_c')$ is continuous.
\end{propos}
As a straightforward corollary, we deduce for instance that if $c_n \to c$ and $x_n \to x$ and if  the $\cu'_{c_n}(x_n)$ exist, as well as $\cu'(c)(x)$, then $\cu'_{c_n}(x_n) \to \cu'(c)(x)$.
 
 
%

%

\begin{proof}[Proof of Proposition \ref{discount}]
If $\rho(c_0)\in \R\setminus \Q$, there is a unique weak K.A.M. solution up to constants, hence continuity of $\mathcal{PG} ( \cu_c')$ at $c_0$ follows from Proposition \ref{hausdorff}.

If $\rho(c) = r\in \Q$, let us denote $\rho^{-1}(r) = [c_1,c_2]$. Again, continuity at $c_1$ and $c_2$ is obvious as there is a unique weak K.A.M solution at these cohomology classes (see Proposition \ref{unicite}).

It remains to study what happens inside $(c_1,c_2)$ and we will prove that in this interval, the map $c\mapsto \cu_c$ is concave. Let us set $\cm_r $ the set of Mather measures corresponding to any cohomology class $c\in (c_1,c_2)$. Recall that as seen in \eqref{Matherinv} page \pageref{Matherinv}, this set does not depend on $c$. Moreover, the function $\alpha$ is affine on $(c_1,c_2)$. 

From \cite{DFIZ2}, we know that $\cu_c(x) = \sup_{u} u(x)$, where the supremum is taken amongst (continuous) $c$-dominated functions $u:\T\to \R$ such that $\int u(x) d\mu(x,y) \leqslant 0$ for all $\mu \in \cm_r$. Moreover, it is proven that $\int \cu_c(x) d\mu(x,y) \leqslant 0$  for all $\mu \in \cm_r$. Let now $c,c' \in (c_1,c_2)$ and $\lambda \in [0,1]$. Let us set $v= \lambda \cu_c + (1-\lambda)\cu_{c'}$.

As $\int \cu_c(x) d\mu(x,y) \leqslant 0$ and $\int \cu'_c(x) d\mu(x,y) \leqslant 0$   for all $\mu \in \cm_r$ we deduce that $\int v(x) d\mu(x,y) \leqslant 0$  for all $\mu \in \cm_r$.

Moreover, passing to lifts (with the same $\sim$ notation as previously), from

$$\forall \theta, \theta'\in \R, \quad  \widetilde \cu_c(\theta)-\widetilde \cu_c(\theta')\leqslant S(\theta', \theta)+c(\theta'-\theta)+ \alpha(c);$$
$$\forall \theta, \theta'\in \R, \quad  \widetilde \cu_{c'}(\theta)-\widetilde \cu_{c'}(\theta')\leqslant S(\theta', \theta)+c'(\theta'-\theta)+ \alpha(c');$$
and recalling that $\alpha\big(\lambda c + (1-\lambda)c'\big) = \lambda \alpha(c) + (1-\lambda)\alpha(c')$, we get

$$\forall \theta, \theta'\in \R, \quad  \tilde v(\theta)-\tilde v(\theta')\leqslant S(\theta', \theta)+\big(\lambda c + (1-\lambda)c'\big)(\theta'-\theta)+\alpha\big(\lambda c+(1-\lambda) c'\big).$$
Hence $v$ is $\big(\lambda c+(1-\lambda) c'\big)$-dominated. We conclude that $v\leqslant \cu_{\lambda c + (1-\lambda) c'}$, proving the claim, and the Proposition.

\end{proof}
 \begin{remk}
 The previous proof is intimately linked to the $1$-dimensional setting we work with. Indeed, it was communicated to us by Patrick Bernard that as soon as we move up to dimension $2$, there are examples on $\T^2$ for which it is not possible to construct a function $c\mapsto u_c$ that maps to each cohomology class a weak K.A.M. solution and such that $c\mapsto u'_c$ is continuous (in any possible way). 
 \end{remk}
We obtain as a corollary:

\begin{cor}
The function $\cu(x,c) = \cu_c(x) - \cu_c(0)$ also satisfies the conclusions of Theorem \ref{Tgenecont}.
\end{cor}

We now  give a $C^\infty$ integrable example for which the discounted method doesn't select a transversely continuous weak K.A.M. solution. 
 
 \begin{exa}
 We use the notation of Theorem \ref{Tgeneder}. We define $F_0, H:\A\rightarrow \A$ by $F_0(\theta, r)=(\theta +r, r)$ and $H(\theta, r)=(h(\theta), \frac{r}{h'(\theta)})$ where $h:\T\rightarrow \T$ is a smooth orientation preserving diffeomorphism of $\T$ such that $h(t)=t+d(t)$ and  $d:\T\rightarrow \R$ satisfies $d(0)=0$ and 
 \begin{equation}\label{Edemi}\int_\T d(t)dt>\frac{d(\frac{1}{2})}{2}.
 \end{equation}  Observe that $h^{-1}(t)=t-d\circ h^{-1}(t)$.  As the symplectic diffeomorphism $H$ maps a vertical $\{ \theta\}\times \R$ onto a vertical $\{ h(\theta)\}\times \R$ and preserves the transversal orientation, the smooth diffeomorphism\footnote{{Note that $F_0$ is the time-$1$ map of the Hamiltonian function $f_0(\theta, r) = \frac 12 r^2$. It follows that $F$, being conjugated to $F_0$ by a symplectic map, is itself the time-$1$ map of the Tonelli Hamiltonian $f_0\circ H^{-1}$. }}  $F=H\circ F_0\circ H^{-1}$ is also a symplectic  $C^\infty$ integrable ESTwD. The new invariant foliation is the set of the graphs of $\eta_c(\theta)=\frac{c}{h'\textrm{$\big($}h^{-1}(\theta)\textrm{$\big)$}}=c(h^{-1})'(\theta)$. Hence we have $u_c(\theta)=-cd\circ h^{-1}(\theta)$. Observe that the function $u$ is smooth.\\
Then $H_c(\theta)=\theta+\frac{\partial u_c}{\partial c}(\theta)=\theta-d\circ h^{-1}(\theta)=h^{-1}(\theta)$. Hence the measure defined on $\T$ by $\mu ([0, \theta])=h^{-1}(\theta)$, i.e. the measure with density $\frac{1}{h'\circ h^{-1}}$, is invariant by the restricted-projected Dynamics $g_c$. When the rotation number $ \rho(c)$ of $g_c$ is irrational, this is the only measure invariant by $g_c$.\\
 Let us recall that the discounted solution $\cu_c$ that is selected in \cite{SuThi} and  \cite{DFIZ2} is the weak K.A.M. solution that is the supremum of the subsolutions that satisfy for every   minimizing $g_c$-invariant measure $\mu$: $\int u_cd\mu\leqslant 0$. When $c$ is irrational, we deduce that
 $$\cu_c(\theta)=u_c(\theta)-\int u_c(t)d\mu(t)=c\left(\int_\T  d\circ h^{-1}(t)(h^{-1})'(t)dt -d\circ h^{-1}(\theta)\right);$$
 i.e. 
\begin{equation}\label{Eirrat} \cu_c(\theta)=c\left(\int_\T d(t)dt-d\circ h^{-1}(\theta)\right)=u_c(\theta)+c\int_\T d(t)dt.\end{equation}
  Assume now that $c=\frac{1}{2}$. Then 
  $$g_{\frac{1}{2}}(0)=h\circ R_{\frac{1}{2}}\circ h^{-1}(0)=h\Big(\frac{1}{2}\Big)=\frac{1}{2}+d\Big(\frac{1}{2}\Big)\quad{\rm and}\quad g_{\frac{1}{2}}\left( \frac{1}{2}+d\Big(\frac{1}{2}\Big)\right)=0.$$
  The mean of the two Dirac measures
  $$\nu=\frac{1}{2}\left( \delta_0+\delta_{\frac{1}{2}+d(\frac{1}{2})}\right)$$
  is a measure that is invariant by $g_\frac{1}{2}$. Hence $\cu_{\frac{1}{2}}(\theta)=u_{\frac{1}{2}}(\theta)-K$ with $K\geqslant\int_\T u_{\frac{1}{2}}d\nu$. We deduce that $$K\geqslant \frac{1}{2}\left( u_{\frac{1}{2}}(0)+u_{\frac{1}{2}}\bigg(\frac{1}{2}+d\Big(\frac{1}{2}\Big)\bigg)\right)=-\frac{1}{4}\left( d\circ h^{-1}(0)+d\circ h^{-1} \bigg(\frac{1}{2}+d\Big(\frac{1}{2}\Big)\bigg)\right);$$
  i.e.
  $$K\geqslant -\frac{1}{4}d\Big(\frac{1}{2}\Big).$$
By  Inequality  (\ref{Edemi}), we know that $\varepsilon=\int_\T d(t)dt-\frac{d(\frac{1}{2})}{2}>0$.  We have then 
  $$\cu_{\frac{1}{2}}(\theta)\leqslant u_{\frac{1}{2}}(\theta)+ \frac{1}{4}d\Big(\frac{1}{2}\Big)= u_{\frac{1}{2}}(\theta)+\frac{1}{2}\int_\T d(t)dt-\frac{\varepsilon}{2}
  $$
  Using Equation (\ref{Eirrat}), we deduce that
  $$\limsup_{c\rightarrow\frac{1}{2}}\cu_c(\theta)\geqslant \cu_{\frac{1}{2}}(\theta)+\frac{\varepsilon}{2}.$$
  Hence $(\theta, c)\mapsto \cu_c(\theta)$ is not continuous.\\
  Observe that in the integrable case, there exists a unique weak K.A.M. solution in each cohomology class up to the addition of a constant. Hence selecting a weak K.A.M. solution in every cohomology class is reduced in this case to choosing a constant. Using this remark, it can be proved that for the integrable case, the discounted choice is lower semi-continuous.
 \end{exa}

 \subsection{An example of weak K.A.M. solution with a calibrating orbit starting from the interior of a vertical bar}\label{SSJPMarco}
 
 We have seen that for an ESTwD $f$, if   $u_c$ is a weak K.A.M. solution  for $\widehat T^c$ and $(\tilde\theta_k)_{k\in \Z_-}$  calibrates its lift $\tilde u_c$, then setting  for $k\in \mathbb Z_-$, $r_k = \frac{\partial S}{\partial \widetilde\Theta} (\tilde\theta_{k-1},\tilde\theta_k)$, the sequence $(\tilde\theta_k,r_k)_{k\in \Z_-}$ is a piece of orbit of $F$ such that $(\tilde \theta_0 , r_0) \in \pg(c+\tilde u'_c)$ and for all $k<0$, $(\tilde \theta_k , r_k) \in \cg(c+\tilde u'_c)$. We now construct an example of such a situation where  $(\tilde \theta_0 , r_0) \notin  \overline{\cg(c+\tilde u'_c)}$. It can be proven that such a situation cannot happen if $f$ is the time--$t$ map of an autonomous Tonelli Hamiltonian flow, for any $t>0$. 

Let us start from the classical pendulum Hamiltonian $H : T^*\T \to \R$ defined by 
$$\forall (\theta,p)\in  T^*\T , \quad H(\theta,p) = \frac12 |p|^2 + \cos(2\pi\theta).$$

Let $s^+ : \theta \mapsto \sqrt{2-2\cos(s\pi\theta)}$ be the function whose graph, $\mathcal S^+$, is the upper part of the level set $H^{-1}(\{1\})$ and $c_0 = \int_0^{1} s^+(\theta) d \theta$. Finally, let $t_0>0$ be a small enough real number. It is then known that if $\Phi_H$ denotes the Hamiltonian flow of $H$, for $t_0$ small enough, $\Phi_H^{t_0}$ is an ESTwD that we denote by $f_0$. We also denote by $S_0 : \R^2\to \R$ a generating function associated to the lift  $F_0 : \R^2 \to \R^2$ of $f_0$ that fixes $(0,0)$. It can be proven that at cohomology $c_0$, there is a unique weak K.A.M. solution $u_0$ for  $\widehat T^{c_0}$ such that $u_0(0) = 0$ and it is given by 

$$\forall \theta \in \R, \quad u_0(\theta) = \int_0^\theta s^+(t) dt -c_0\theta .$$
This function is $C^1$ and $\cg(c_0+u'_0) = \pg(c_0+u'_0) = \mathcal S^+$. Moreover, $\mathcal S^+$ is invariant by $f_0$.

The dynamics of $f_0$ restricted to $ \mathcal S^+\backslash\{ (0, 0)\}$ is going from the left to the right with a fixed point $(0,0) = (1,0)$. Let  $[a_{-1},a_0)\subset (0,1)$  a fundamental domain of the projected dynamics restricted to $ \mathcal S^+\backslash\{ (0, 0)\}$. This means that if $\mathcal S^+_{|[a_{-1},a_0)} = \big\{\big(\theta, s^+(\theta)\big), \ \theta \in [a_{-1},a_0) \big\}$, then $\mathcal S^+ $ is the disjoint union of $\{(0,0)\}$ and of the $f_0^n(\mathcal S^+_{|[a_{-1},a_0)})$ when $n\in \Z$. In particular, $f_0\big(a_{-1}, s^+(a_{-1})\big) = (a_0, s^+(a_0)\big)$.

Let  $\varphi : \T \to [0,+\infty)$ be a $C^2$ function supported in $[a_{-1},a_0]$, we define the diffeomorphism $v_\varphi : \A \to \A$ by $(\theta,r)\mapsto \big( \theta, r+\varphi'(\theta)\big)$ and then $f_\varphi = v_\varphi \circ f_0$ that is also an EStwD. A direct computation shows that $S_\varphi : \R^2 \to \R$, defined by $(\tilde \theta, \wt)\mapsto S_0(\tilde \theta, \wt)+\varphi(\wt) $ is the generating function of $F_\varphi$,  the lift of $f_\varphi$ that fixes $(0,0)$ (we still denote by $\varphi : \R \to \R$ the lift of $\varphi : \T \to \R$). As $\varphi\geqslant 0$, it follows that $S_\varphi \geqslant S_0$.

For $F_0$, the projected Mather set at cohomology $c_0$ is  $\{k,\ \ k\in \Z\}$, as $(k,0) $, $k\in \Z$ are the only fixed points of $F_0$ in $\mathcal S^+$. Also, the rotation number at cohomology $c_0$ is $0$. We deduce that if $k\in \Z$, then $S_0(k,k) = \min\limits_{y\in \R} S_0(y,y)$ by \cite{Ban}.
It follows that for $F_\varphi$, the projected Mather set for the $0$ rotation number is also $\{k, \ \ k\in \Z\}$, by \cite{Ban}, as if $x\in \R$ is in this projected Mather set, then $S_\varphi(x,x) = \min\limits_{y\in \R} S_\varphi(y,y) =  \min\limits_{y\in \R} S_0(y,y)$. 

Let now $c_\varphi\in \R$ be the biggest cohomology class such that $\rho(c_\varphi) = 0$ for $F_\varphi$. Let $\tilde u_\varphi : \R \to \R$ be the corresponding weak K.A.M. solution at cohomology $c_\varphi$ such that $\tilde u_\varphi(0)=0$. In the following lemmas, we study properties of $\tilde u_\varphi$. We introduce $a_{-2}\in [0,a_{-1})$ such that $F_0\big(a_{-2},c_0+\tilde u'_0(a_{-2})\big) = \big(a_{-1},c_0+\tilde u'_0(a_{-1})\big)$.

\begin{lemma}
The function $\tilde u_\varphi$ is $C^1$ on $[0,a_{-1}]$ and  equality $(c_0+\tilde u'_0)_{|[0,a_{-1}]}=(c_\varphi+\tilde u'_\varphi)_{|[0,a_{-1}]}$ holds.
\end{lemma}
\begin{proof}
Let $\tilde \theta_0 \in [0,a_{-1}]$ and let $(\tilde \theta_k)_{k\leqslant 0}$ be the unique minimizing chain starting at $\tilde \theta_0$ that calibrates $\tilde u_0$ (for $S_0$). From what was recalled above, $\tilde \theta_k \to 0$ and is non--decreasing with $k$. Moreover, as the $\theta_k$'s are not in the support of $\varphi$, the sequence $(\tilde \theta_k)_{k\leqslant 0}$ is also minimizing for $S_\varphi$ (because $S_\varphi \geqslant S_0$). Hence, by Proposition \ref{minrat}, $(\tilde \theta_k)_{k\leqslant 0}$ calibrates $\tilde u_\varphi$ (for $S_\varphi$). It follows that for $k\leqslant -1$, $\tilde u_\varphi$ is derivable at $\tilde\theta_k$ and 
$$c_\varphi + \tilde u'_\varphi(\tilde\theta_k) = \frac{\partial S_\varphi}{\partial \widetilde\Theta} (\tilde\theta_{k-1},\tilde\theta_k) = \frac{\partial S_0}{\partial \widetilde\Theta} (\tilde\theta_{k-1},\tilde\theta_k)=c_0+\tilde u'_0 (\tilde\theta_k).$$
When $\tilde\theta_0$ sweeps $[0,a_{-1}]$, $\tilde\theta_{-1}$ takes all values in $[0,a_{-2}]$.

We then extend what was just established to $(a_{-2},a_{-1}]$. Let $\tilde \theta_0 \in (a_{-2},a_{-1}]$ a point where $\tilde u_\varphi $ is derivable. The previous argument shows that if $(\tilde \theta_k)_{k\leqslant 0}$ is the unique calibrating chain for $u_\varphi$, then is is also the unique calibrating chain for $\tilde u_0$. Hence using the previous result,
$$\big(\tilde \theta_0, c_\varphi + \tilde u'_\varphi(\tilde\theta_0)\big) = F_\varphi\big(\tilde \theta_{-1}, c_\varphi + \tilde u'_\varphi(\tilde\theta_{-1})\big) = F_0\big(\tilde \theta_{-1}, c_0 + \tilde u'_0(\tilde\theta_{-1})\big) = \big(\tilde \theta_0, c_0 + \tilde u'_0(\tilde\theta_0)\big).$$
In the above, we used the fact that $\tilde \theta_{-1} \in [a_{-2},a_{-1}]$ lies away from the support of $\varphi$ and then $F_0$ and $F_\varphi$ coincide on the fiber above $\tilde \theta_{-1}$.
As a conclusion, restricted to $ [a_{-2},a_{-1}]$, the Lipschitz functions $\tilde u_\varphi$ and  $t\mapsto \tilde u_0 + (c_0-c_\varphi)t$ have the same derivative almost everywhere and same value at $a_{-2}$, then they are equal.

\end{proof}

The next lemma provides the values of $\tilde u_\varphi$ on $(a_{-1},a_0]$.

\begin{lemma}
The function $\tilde u_\varphi$ is $C^1$ on  $(a_{-1},a_0]$ and for all $t\in  (a_{-1},a_0]$, $c_\varphi +\tilde u'_\varphi (t) = c_0+\tilde u'_0(t) + \varphi'(t)$.
\end{lemma}

\begin{proof}
Let us now consider the chain $(a_{-2},a_{-1},a_0)$, that calibrates $\tilde u_0$ and is minimizing for $S_0$.
By the same arguments used in the previous lemmas, the same chain $(a_{-2},a_{-1},a_0)$ is also minimizing for $S_\varphi$ and calibrates $\tilde u_\varphi$. It follows from Lemma \ref{orderWK.A.M.} that if $(\tilde\theta_k)_{k\leqslant 0}$ calibrates $\tilde u_\varphi$, and if $\tilde \theta_0 \in(a_{-1},a_0)$, then $\tilde\theta_{-1} \in (a_{-2},a_{-1})$. If now $\tilde\theta_0$ is a derivability point of $\tilde u_\varphi$, then $\big(\tilde\theta_0 , c_\varphi + \tilde u'_\varphi (\theta_0)\big) = F_\varphi\big(\tilde\theta_{-1} , c_\varphi + \tilde u'_\varphi (\theta_{-1})\big) \in \cg ( c_0 + \tilde u'_0 + \varphi')_{| (a_{-1},a_0)}$.
Indeed, $F_\varphi\big( \cg ( c_0 + \tilde u'_0)_{| (a_{-2},a_{-1})}\big) = v_\varphi\big( \cg ( c_0 + \tilde u'_0)_{| (a_{-1},a_0)}\big) =  \cg ( c_0 + \tilde u'_0 + \varphi')_{| (a_{-1},a_0)}$.

We conclude that $ \cg ( c_\varphi + \tilde u'_\varphi)_{| (a_{-1},a_0)}\subset  \cg ( c_0 + \tilde u'_0 + \varphi')_{| (a_{-1},a_0)}$. As previously, this implies that the inclusion must be an equality and this proves the lemma.
\end{proof}

We now specify how to chose $\varphi$ in order to obtain our example:\\
{\bf Hypothesis:} Let $a_1>a_0$ such that $F_0\big(a_0,c_0+\tilde u'_0(a_0)\big) = \big(a_1,c_0+\tilde u'_0(a_1)\big)$. We assume that $\varphi$ is chosen as follows: there exists $d\in(a_0,a_1)$ such that  $F_\varphi\big( \cg ( c_\varphi + \tilde u'_\varphi)_{| (a_{-1},a_0)}\big) = F_0\big( \cg ( c_\varphi + \tilde u'_\varphi)_{| (a_{-1},a_0)}\big)$ is the union of a graph above $(a_0,d)$,  a graph above $ (d,a_1)$ and  a non trivial vertical interval above $\{d\}$.

\begin{figure}[h!]
\begin{center}

 \begin{tikzpicture}[scale = 1.2]
    \draw[color=red][line width=2pt][domain=0:4][samples=250] plot (2*\x,3* abs{sin(pi*\x /4 r)} );
   
   \draw  (0,-4)--(0,4) ;
     \draw  (8,-4)--(8,4) ;
     \draw (0,0) -- (8,0);

\draw (0,0) node[left]{0} ;
\draw (8,0) node[right]{1} ;

\draw  (4,-.1)--(4,.1) ;
\draw (4,0) node[below]{$a_0$} ;

\draw  (6,-.1)--(6,.1) ;
\draw (6,0) node[below]{$a_1$} ;

\draw  (2,-.1)--(2,.1) ;
\draw (2,0) node[below]{$a_{-1}$} ;

\draw  (1,-.1)--(1,.1) ;
\draw (1,0) node[below]{$a_{-2}$} ;

  \draw[color=blue][line width=1pt][domain=1.25:1.75][samples=250] plot (2*\x,{3* abs{sin(pi*\x /4 r)} -500*(\x-1.25)*(\x-1.25)*(\x-1.75)*(\x-1.25)*(\x-1.75) -500*(\x-1.25)*(\x-1.75)*(\x-1.75)*(\x-1.25)*(\x-1.75) } );
  
    \draw[color=blue][line width=1pt][domain=2.25:2.5][samples=250] plot (2*\x,{3* abs{sin(pi*\x /4 r)} +5*(\x-2.25)*(\x-2.25)*(sqrt(2.5-\x)+1) } );
    
       \draw[color=blue][line width=1pt][domain=2.5:2.75][samples=250] plot (2*\x,{3* abs{sin(pi*\x /4 r)} -5*(\x-2.75)*(\x-2.75)*(sqrt(\x-2.5)+1) } );

 \draw[color=blue][line width=1pt] (5, {3* abs{sin(pi*2.5 /4 r)} +5*(2.5-2.25)*(2.5-2.25)*(sqrt(2.5-2.5)+1) }) -- (5, {3* abs{sin(pi*2.5 /4 r)} -5*(2.5-2.25)*(2.5-2.25)*(sqrt(2.5-2.5)+1) });

\draw  (5,-.1)--(5,.1) ;
\draw (5,0) node[below]{$d$} ;

\end{tikzpicture}

\caption{In red the graph of $c_0+u'_0$. In blue, the perturbation giving $c_\varphi + u'_\varphi$ on $[0, a_1]$.}
\label{double1}
\end{center}
\end{figure}
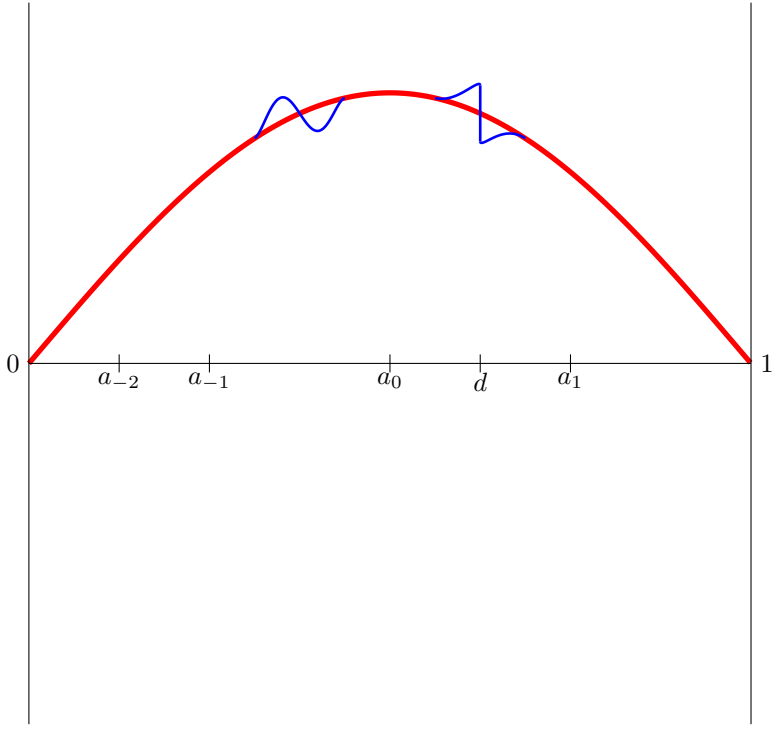

The next lemma provides a description of the weak K.A.M. solution for $S_\varphi$ on $(a_0,a_1)$.

\begin{lemma}
Under the previous hypothesis, 
$$\pg(c_\varphi +u'_\varphi)_{|(a_0,a_1)} = F_\varphi\big( \cg ( c_\varphi + \tilde u'_\varphi)_{| (a_{-1},a_0)}\big).$$
\end{lemma}

\begin{proof}
Arguing as in the two previous lemmas, we find that $\cg(c_\varphi +u'_\varphi)_{|(a_0,a_1)} \subset F_\varphi\big( \cg ( c_\varphi + \tilde u'_\varphi)_{| (a_{-1},a_0)}\big)$. As the right hand side set is a graph above  $(a_0,d)\cup (d,a_1)$, again arguing as previously we obtain that $u_\varphi$ is $C^1$ on  $(a_0,d)\cup (d,a_1)$ and that $\cg(c_\varphi +u'_\varphi)_{|(a_0,d)\cup (d,a_1)}  =F_\varphi\big( \cg ( c_\varphi + \tilde u'_\varphi)_{| (a_{-1},a_0)}\big) \setminus  \{d\}\times \R $.
The result follows.
\end{proof}

Next we prove that this construction indeed yields the desired example. To that end, we prove that any negative orbit of $F_\varphi$ starting from the vertical bar of $ \pg ( c_\varphi + \tilde u'_\varphi)$ above $d$ calibrates $\tilde u_\varphi$.

\begin{propos}
Let  $(\tilde\theta_0,r_0)\in  \pg ( c_\varphi + \tilde u'_\varphi)\cap \{d\}\times \R$ and $(\tilde \theta_{-1},r_{-1}) = F_\varphi^{-1}(\tilde\theta_0,r_0)$. Then 
$$\tilde u_\varphi (\tilde\theta_0) - \tilde u_\varphi(\tilde\theta_{-1}) = S_\varphi(\tilde\theta_{-1},\tilde\theta_0) + c_\varphi(\tilde\theta_{-1}-\tilde\theta_0)+\alpha_\varphi(c_\varphi),
$$
where $\alpha_\varphi$ is Mather's function associated to $F_\varphi$.
\end{propos}

\begin{proof}
If $\tilde\theta_0\in (a_0,a_1)$ and $\tilde\theta_0 \neq d $, then $\tilde u_\varphi$ is derivable at $\tilde\theta_0$. Setting 
$(\tilde \theta_{-1},r_{-1}) = F_\varphi^{-1}\big(\tilde\theta_0,c_\varphi + \tilde u'_\varphi(\tilde\theta_0)\big)$, then 
$$\tilde u_\varphi (\tilde\theta_0) - \tilde u_\varphi(\tilde\theta_{-1}) = S_\varphi(\tilde\theta_{-1},\tilde\theta_0) + c_\varphi(\tilde\theta_{-1}-\tilde\theta_0)+\alpha_\varphi(c_\varphi),$$
using classical results on weak K.A.M. solutions recalled page \pageref{calib} \big(see Equation \eqref{calib}\big). 

Let now $[\tilde\theta(0),\tilde\theta(1)]  = \pi_1\big((\pi_1\circ F_{\varphi | \cg ( c_\varphi + \tilde u'_\varphi)} )^{-1}(\{d\}) \big) \subset (a_{-1},a_0)  $. For $t\in [0,1]$ we define $\tilde\theta(t) = (1-t)\tilde \theta(0)+t\tilde\theta(1)$ and $\big(d,R(t)\big) = F_{\varphi}\big( \tilde\theta(t), c_\varphi + \tilde u'_\varphi\big(\tilde\theta(t)\big)\big)$. As for $i\in \{0,1\}$, $ F_\varphi\big( \tilde\theta(i), c_\varphi + \tilde u'_\varphi\big(\tilde\theta(i)\big)\big)\in \overline{ \cg ( c_\varphi + \tilde u'_\varphi)}$, equality
$$\tilde u_\varphi (d) - \tilde u_\varphi\big(\tilde\theta(i)\big) = S_\varphi(\tilde\theta(i),d )+ c_\varphi(\tilde\theta(i)-d)+\alpha_\varphi(c_\varphi),
$$
still holds. Let now $t\in (0,1)$, one computes using the definition of the generating function $S_{\varphi} $ (see Equations \eqref{genfundef}) that

\begin{multline*}
\tilde u_\varphi (d) - \tilde u_\varphi\big(\tilde\theta(t)\big)- c_\varphi(\tilde\theta(t)-d)  \\
=  \tilde u_\varphi (d) - \tilde u_\varphi\big(\tilde\theta(0)\big)- c_\varphi(\tilde\theta(0)-d)         -\int_0^t \big[\tilde u'_\varphi \big(\tilde\theta(s)\big) + c_\varphi\big]\theta'(s) ds \\
= S_\varphi(\tilde\theta(0),d )+\alpha_\varphi(c_\varphi) + \int_0^t \frac{\partial S}{\partial \tilde\theta}(\tilde\theta(s),d)\tilde\theta'(s) ds \\
 = S_\varphi(\tilde\theta(t),d )+\alpha_\varphi(c_\varphi) .
\end{multline*}
This finishes the proof.
\end{proof}

In order to conclude, we explain how to construct the function $\varphi $ as desired. In fact, we rather construct  $F_\varphi\big( \cg ( c_\varphi + \tilde u'_\varphi)_{| (a_{-1},a_0)}\big)$. To that end, let $\varepsilon_0>0$ be a small real number to be specified through the construction. Let $d\in (a_0,a_1)$ and assume $\varepsilon_0<\min(d-a_0, a_1-d)$. We set $R = \{(\tilde\theta, s^+(\tilde\theta)+r), \ \ \theta\in [d-\varepsilon_0,d+\varepsilon_0], |r|\leqslant \varepsilon_0\}$. Let $\Psi_0 : R\mapsto  [d-\varepsilon_0,d+\varepsilon_0]\times[-\varepsilon_0,\varepsilon_0]$ defined by  $\Psi_0(\tilde \theta, r)=(\tilde\theta, r-s^+(\tilde\theta))$.
 Obviously, $\Psi_0$ preserves each vertical $V_{\tilde\theta}$.

Let $\tilde \theta \in (a_{-1},a_0)$, and let $\tilde\theta_0 \in (a_0,a_1)$ such that $F_0(V_{\tilde\theta})$ and $\mathcal S^+$ intersect at $\big(\tilde\theta_0,s^+(\tilde\theta_0)\big)$. As $\mathcal S^+$ is $F_0$ invariant and $F_0$ is a twist map, by \cite{Arna1}, at this intersection point, the slope of $F_0(V_{\tilde\theta})$ is greater than $(s^{+})'(\tilde\theta_0)$. If $(\tilde\theta_0,r) \in  [d-\varepsilon_0,d+\varepsilon_0]\times[-\varepsilon_0,\varepsilon_0]$, and $\tilde\theta$ is the unique real number such that $(\tilde\theta_0,r)    \in \Psi_0\circ F_0 ( V_{\tilde\theta})$ let $v_1(\tilde\theta_0,r)$ be the slope of $\Psi_0\circ F_0 ( V_{\tilde\theta})$ at $(\tilde\theta_0,r)$. Then up to taking $\varepsilon_0$ smaller, by the previous fact and continuity, we may assume that  $v_1(\tilde\theta_0,r)>0$ for all $(\tilde\theta_0,r) \in  [d-\varepsilon_0,d+\varepsilon_0]\times[-\varepsilon_0,\varepsilon_0]$. Let then $\varepsilon_1>0$ such that $v_1(\tilde\theta_0,r)>\varepsilon_1$ for all $(\tilde\theta_0,r) \in  [d-\varepsilon_0,d+\varepsilon_0]\times[-\varepsilon_0,\varepsilon_0]$.

Finally, let $\varepsilon_2>0$ and $\Psi_1 : (\tilde\theta,r)\mapsto (\tilde\theta-\varepsilon_2r,r)$. Each vertical is sent by $\Psi_1$ to a straight line of slope $-\varepsilon_2^{-1}$. We assume that $\varepsilon_2$ is chosen small enough so that in $\Psi_1([d-\varepsilon_0,d+\varepsilon_0]\times[-\varepsilon_0,\varepsilon_0])$ the curves $\Psi_1\circ \Psi_0\circ F_0 ( V_{\tilde\theta})$ still are graphs of increasing functions with derivative greater than $\varepsilon_1$.

Let now $\rho : \R \to \R$ be the $C^\infty$ function supported in $[-1,1]$ defined by 
$$\forall x\in [-1,1],\quad \rho(x) = \left(\int_{-1}^1 \exp[(s^2-1)^{-1}]ds \right)^{-1} \exp[(x^2-1)^{-1}],
$$
and if $\varepsilon>0$ we define $\rho_\varepsilon : x\mapsto \varepsilon^{-1}\rho(\varepsilon^{-1}x)$.
For $s>0$ small  enough, we define the function $g_{s}$ as the  continuous piecewise affine function that vanishes outside of $ [d-\varepsilon_0+s,d+\varepsilon_0-s]$ and  that is $x\mapsto \frac{d-x}{\varepsilon_2}$ for $x\in[d-s,d+s]$ and that is affine on each remaining connected component of $\R$.
Finally, we set $h_s =\rho_{s/2}* g_s $ where $*$ stands for the regular convolution product.

There exists $\varepsilon_3>0$ such that for $s<\varepsilon_3$ the following hold:
\begin{enumerate}
\item $h_s$ is $C^{\infty}$ and well defined and the non--vanishing part of its graph is included in $\Psi_1([d-\varepsilon_0,d+\varepsilon_0]\times[-\varepsilon_0,\varepsilon_0])$,
\item $h_s$ coincides with $x \mapsto  \frac{d-x}{\varepsilon_2}$ on $[d-s/2, d+s/2]$ and has derivative greater than $ -\varepsilon_2^{-1}$ elsewhere,
\item $h_s' <\varepsilon_1$ on $\R$.
\end{enumerate}
The second point implies that $(\Psi_1\circ \Psi_0)^{-1}\big(\cg(h_{s|[a_0,a_1]})\big) $ coincides with $V_d $ on a non trivial segment, it coincides with $\mathcal S^+$ on neighborhoods of $a_0$ and $a_1$, and it is the graph of a smooth function apart for the vertical part on $V_d$.

The third point implies that  $F_0^{-1}\Big( (\Psi_1\circ \Psi_0)^{-1}\big(\cg(h_{s|[a_0,a_1]})\big)\Big) $ is transverse to the vertical foliation. Hence it is the graph of a smooth function $(s^++\varphi_s) $ where $\varphi_s$ is supported on $[a_{-1},a_0]$. We now wish to set $\varphi ( \tilde\theta) =  \int_{a_{-1}}^{\tilde\theta} \varphi_s(t)dt$. The problem is that there is a priori no reason that $\int_{a_{-1}}^{a_0} \varphi_s(t)dt = 0$ so that $\varphi$ would be compactly supported.

To remedy this, we slightly modify our construction. If $s,s' < \varepsilon_3$ we set $h_{s,s'} = h_s$ on $[a_0, d]$ and $h_{s,s'} = h_{s'}$ on $[d,a_1]$. Again, it coincides with $V_d $ on a non trivial segment, it coincides with $\mathcal S^+$ on neighborhoods of $a_0$ and $a_1$, and it is the graph of a smooth function apart for the vertical part on $V_d$. Then, $F_0^{-1}\Big( (\Psi_1\circ \Psi_0)^{-1}\big(\cg(h_{s|[a_0,a_1]})\big)\Big) $ is the graph of a smooth function  $(s^++\varphi_{s,s'}) $ where $\varphi_{s,s'}$ is supported on $[a_{-1},a_0]$. Now, by the intermediate value theorem, it is possible, given $s$ small, to find $s'$ such that $\int_{a_{-1}}^{a_0} \varphi_{s,s'}(t)dt = 0$ and defining $\varphi ( \tilde\theta) =  \int_{a_{-1}}^{\tilde\theta} \varphi_{s,s'}(t)dt$ on $[a_{-1},a_0]$ we obtain the desired function.

 \begin{remk}
 \begin{enumerate}
 \item In the constructed example, all backward orbits starting on the vertical bar of $ \pg ( c_\varphi + \tilde u'_\varphi)$ above $d$ calibrate the weak K.A.M. solution. Modifying slightly the example is is also possible to have a unique backward orbit starting on the vertical bar of $ \pg ( c_\varphi + \tilde u'_\varphi)$ above $d$ calibrate the weak K.A.M. solution.
 \item The same construction can be made, starting from an invariant circle of arbitrary rotation number.
 \end{enumerate}

 \end{remk}

 \section{Some results concerning the full pseudographs}\label{Apfulpseudo}

 Most of the results that follow are standard and even hold in all dimension. One can find them in similar of different formulations in \cite{cansin}. However, we provide proofs for the reader's convenience.

 \subsection{An equivalent definition}
\hglue 14truecm 
 
 \begin{defi}
 Let $u:\R\rightarrow \R$ be a $K$ semi-concave function. Then $p\in\R$ is a {\em super-derivative} of $u$ at $x\in\R$ if
$$\forall y\in \R,\quad u(y)-u(x)-p(y-x)\leqslant \frac{K}{2}(y-x)^2.$$
We denote the set of super-derivatives of $u$ at $x$ by $\partial^+u(x)$.   It is a convex set.
\end{defi}
Observe that a derivative is always a super-derivative.  If $u:\R\rightarrow \R$ is $K$-semi-concave, then  $x\mapsto u(x)-\frac{K}{2}x^2$ is concave and thus  locally Lipschitz, and $x\mapsto u'(x)-Kx$ is non-increasing. Hence a $1$-periodic $K$-semi-concave function is $K$-Lipschitz.

 Observe also that $\displaystyle{\bigcup_{x\in \T}\{ x\}\times \partial^+u(x)}$ is compact.
\begin{propos}
 Let $u:\R\rightarrow \R$ be a $K$-semi-concave function. Then, for every $x\in\R$, we have 
 $$\partial u (x)=\{ x\}\times \partial ^+u(x).$$
\end{propos} 
 Hence the full pseudograph of $u$ is also the subbundle of all the super-derivatives of $u$.

 \begin{proof}
Let us prove the inclusion $\partial u (x)\subset\{ x\}\times \partial ^+u(x)$.  Let us consider $(x, p)\in\partial u (x)$.  Then there exist $(x,p_-), (x, p_+)\in \overline{\cg(u')}$ such that $p_-\leqslant p\leqslant p_+$ and there exist two sequences $(x_n, p_n), (y_n, q_n)\in\cg(u')$ that respectively converge to $(x, p_-)$, $(x,p_+)$. Every derivative is a super-derivative and a limit of super-derivatives is a super-derivative. Hence, we have $p_-, p_+\in \partial^+u(x)$. By convexity of $\partial^+u(x)$, we deduce that $p\in \partial^+u(x)$.


Let us now prove the reverse inclusion. Being $K$-semi-concave, $u$ is $K$-Lipschitz, hence the set of all its super-derivatives is bounded (by $K$). If $x\in \R$, we have then $\partial^+u(x)=[p_-, p_+]$ with $-K\leqslant p_-\leqslant p_+\leqslant K$. We will prove that $(x,p_-), (x, p_+)\in \partial u(x)$. We have
\begin{multline*}
\forall y\in \R,\quad  u(y)-u(x)-p_-(y-x)\leqslant \frac{K}{2}(y-x)^2
\\
{\rm and}\quad u(y)-u(x)-p_+(y-x)\leqslant \frac{K}{2}(y-x)^2.
\end{multline*}
This implies that
\begin{itemize}
\item for $y>x$, we have 
$$\frac{u(y)-u(x)}{y-x}\leqslant p_-+\frac{K}{2}(y-x);$$
\item for $y<x$, we have 
$$\frac{u(y)-u(x)}{y-x}\geqslant p_++\frac{K}{2}(y-x).$$
\end{itemize}
Recall that $\frac{u(y)-u(x)}{y-x}=\frac{1}{y-x}\int_x^yu'(t)dt$. This gives the existence of two sequences $(x_n)\in (-\infty, x)$ and $(y_n)\in (x, +\infty)$ that converge to $x$ where $u$ is differentiable and
$$\limsup u'(x_n)\geqslant p_+\quad{\rm and}\quad \liminf u'(y_n)\leqslant p_-.$$
As we know that a derivative is a super-derivative, that the set of super-derivatives is closed and that $\partial^+u(x)=[p_-, p_+]$, we deduce that 
$$\big(x, \lim  u'(x_n)\big)=(x,  p_+)\in \partial u(x) \quad{\rm and}\quad \big(x, \lim  u'(y_n)\big)=(x,  p_-)\in \partial u(x).$$
 \end{proof}
 \subsection{Proof of Lemma \ref{Lfullman}}\label{ssLfullman}
 We just recall the argument of the proof of
 \begin{lemma}
For all $c\in \R$, $\mathcal{PG}(c+u'_c)$ is a Lipschitz one dimensional compact manifold that is an essential circle. 
\end{lemma}
\begin{proof}
It is proved in  \cite{Arna2}, that for every $c\in \R$ and every $K$-semi-concave function $u:\T\rightarrow \R$, there exists $\tau>0$ such that $\varphi_{-\tau}\big(\mathcal{PG}(c+u')\big)$ is the graph of a Lipschitz function, where $(\varphi_t)$ is the flow of the pendulum. This gives the wanted result.
\end{proof}

 \subsection{Proof of Proposition \ref{hausdorff}}\label{AppB3}
 Let us now prove the following proposition\footnote{{The statement holds in arbitrary dimension and follows from the same result for concave functions. We present here a simple proof relying on the $1$-dimensional setting.}}. 
 
 \begin{propos}\label{Pconfullps}
Let $(f_n)_{n\in \mathbb N} $ be a sequence of equi-semi-concave functions from $\T$ to $\R$ that converges (uniformly) to a function $f$ (that is hence also semi-concave).

Then $\big(\mathcal{PG}(f'_n) \big)$ converges to $\mathcal{PG}(f')$ for the Hausdorff distance.
\end{propos}

\begin{proof}
Let us prove that the lim sup of the $\mathcal{PG}(f'_n) $ is in $\mathcal{PG}(f')$. Up to a subsequence, we consider $(x_n, p_n)\in \mathcal{PG}(f'_n) $ that converges to some $(x,p)$, and we want to prove that $(x,p)\in \mathcal{PG}(f')$. We have
$$\forall n, \forall y\in \R, \quad f_n(y)-f_n(x_n)-p_n(y-x_n)\leqslant \frac{K}{2}(y-x_n)^2.$$
Taking the limit, we deduce that $(x,p)\in \mathcal{PG}(f')$.

Let us now assume that $\big(\mathcal{PG}(f'_n) \big)$ doesn't converge to $\mathcal{PG}(f')$. There exists a point $(x, p)\in \mathcal{PG}(f')$,  $r>0$ and $N\geqslant 1$ such that, up to a subsequence, 
$$\forall n\geqslant N, \quad\mathcal{PG}(f'_n)\cap B\big((x, p), r\big)=\varnothing.$$
Hence, for $n$ large enough, $\mathcal{PG}(f'_n)$ is contained in a small neighbourhood of a simple arc (and not loop).
This implies that for $n$ large enough, $\mathcal{PG}(f'_n)$ doesn't separate the annulus into two unbounded connected components, a contradiction.

\end{proof}

\section{Sketch of the proof of point \ref{K.A.M.ext} page \pageref{K.A.M.ext}}\label{appendix-3}

We wish to explain why if $u:\cm\big(\rho(c)\big)\rightarrow \R$ is dominated, then there exists only one extension $U$ of $u$ to $\T$ that is a weak K.A.M. solution for $\widehat T^c$ that is given by  
$$\forall x\in \T, \quad U(x) = \inf_{\substack{\pi(\theta)\in  \cm\textrm{$\big($}\rho(c)\textrm{$\big)$}\\ \pi(\theta')=x}} \tilde u(\theta) + \cs^c(\theta,\theta')$$

 where 
$\cs^c(\theta,\Theta)=\inf\limits_{n\in\N} \big(\cs^c_n(\theta, \Theta)+n\alpha(c)\big)$.

\begin{itemize}
\item It is a general fact that if $\pi(\theta)\in \cm\textrm{$\big($}\rho(c)\textrm{$\big)$}$ the function $\theta'\mapsto \cs^c(\theta,\theta')$ is a weak K.A.M solution that vanishes at $\theta'=\theta$ (see \cite[Definition 2.1 and Proposition 2.8]{Za1} recalling that the function $\cs^c$ corresponds to the lift of the Ma\~n\' e potential $\varphi$ in the reference and that our Mather set $\cm\textrm{$\big($}\rho(c)\textrm{$\big)$}$ is included in the Aubry set). As the set of weak K.A.M. is invariant by addition of constants and an infimum of weak K.A.M. solutions is a weak K.A.M. solution (\cite[Lemma 2.33]{Za1}) it follows that $U$ is a weak K.A.M. solution.
\item To prove that $U=u$ on $\cm\textrm{$\big($}\rho(c)\textrm{$\big)$}$ just notice that as $u$ is dominated, if $x\in \cm\big(\rho(c)\big)$ and $\pi(\theta)=x$
$$\forall \theta' \in \pi^{-1}\left(\cm\big(\rho(c)\big)\right),\quad \tilde u(\theta')+\cs^c(\theta',\theta) \geqslant \tilde u(\theta) = u(x) +\cs^c(\theta,\theta).$$
\item It remains to prove that $U$ is unique. This follows from the fact that if two weak K.A.M. solutions $U_1$ and $U_2$  coincide on $\cm\big(\rho(c)\big)$ they are equal.

Let $x_0\in \T$. One constructs inductively a sequence $(x_n)_{n\leqslant 0}$ such that 
$$\forall n< 0, \quad U_1(x_0)=U_1(x_n)+\sum_{k=n}^{-1}S^c(x_k,x_{k+1}).$$
As $U_2$ is a weak K.A.M. (hence dominated) one also has 
$$\forall n< 0, \quad U_2(x_0)\leqslant U_2(x_n)+\sum_{k=n}^{-1}S^c(x_k,x_{k+1}).$$
Hence $U_2(x_0)-U_1(x_0) \leqslant U_2(x_n) - U_1(x_n)$. To conclude, one proves, using a Krylov-Bogoliubov type argument that there exists a subsequence $(x_{\varphi(n)})$ that converges to a point $x\in \cm\big(\rho(c)\big)$, hence proving that $U_2(x_0)-U_1(x_0) \leqslant 0$. Then the result follows by a symmetrical argument.
\end{itemize}

\bibliographystyle{amsplain}

\end{document}